\documentclass{article} 

\usepackage{amsmath}
\usepackage{amssymb}
\usepackage[all]{xy}
\usepackage{graphicx}
\usepackage{array}

\newenvironment{proof}{{\bf Proof.} \  }{{\hspace*{1ex}}\hfill$\Box$} 

\newtheorem{definition}{Definition} 
\newtheorem{proposition}[definition]{Proposition} 
\newtheorem{theorem}[definition]{Theorem}  
\newtheorem{lemma}[definition]{Lemma}   
\newtheorem{corollary}[definition]{Corollary}   
\newtheorem{remark}[definition]{{Remark}}   
\newtheorem{example}[definition]{{Example}}

\newcommand{\calA}{\mathcal{A}} 
 
\newcommand{\calE}{\mathcal{E}}

\newcommand{\Z}{{\mathbb Z}}
\newcommand{\Q}{{\mathbb Q}}
\newcommand{\R}{{\mathbb R}} 
\newcommand{\C}{{\mathbb C}}

\newcommand{\HH}{{\mathbb H}} 
\newcommand{\LL}{{\mathbb L}}

\begin{document}

\begin{center} 
{\bf\LARGE Self-closeness numbers of  finite cell complexes } 

\vspace{5ex} 

{\large  Nobuyuki  Oda\footnote{Department of Applied Mathematics, Faculty of Science, Fukuoka University,   Fukuoka 814-0180,  Japan,  \\ 
e-mail : odanobu@fukuoka-u.ac.jp }   
and 
  Toshihiro Yamaguchi\footnote{Faculty of Education, Kochi University, 2-5-1 Akebono-cho,  Kochi 780-8520, Japan,    \\ 
e-mail :  tyamag@kochi-u.ac.jp } }  
\end{center}

\

{\sc Abstract.} 
We reformulate the inequalities among self-closeness numbers of spaces in cofibrations making use of homology dimension and show that the self-closeness number  of a space is less than or equal to the homology dimension of the space. 
Then  we prove a relation  of self-closeness numbers and  the connectivity for manifolds satisfying Poincar\'{e} duality. 
On the other hand we determine the self-closeness numbers of the real projective spaces,   lens spaces and a cell complex defined by  Mimura and Toda. 
Moreover, making use of the models of Sullivan and Quillen, we show several properties of self-closeness number for  finite cell complexes, and rational examples are studied to obtain some  precise results. 
Finally, we prove relations among self-closeness numbers defined by homotopy groups, homology groups and cohomology groups.

\

\noindent 
keyword  : \  
self-homotopy equivalence, self-closeness number,   cell complex, Sullivan model,  Quillen model \\ 
MSC-class :  55P10;      55P62;   55Q05;   55R05

\

\section{Introdution}

Choi and Lee \cite{ChoiLee15} introduced  the {\it self-closeness number $N\calE (X)$} of a space $X$ by 
$$N\calE (X)  = \min \{\, k \ | \ \calA_{\sharp}^{k} (X) = \calE (X) \},$$   where 
$\calA_{\sharp}^{k} (X)  = \{\, f \in [X,X] \ | \  f_{\sharp} : \pi_{i} (X) \xrightarrow{\cong} \pi_{i} (X) \ \ \textrm{for any} \ \ i \leq  k\}$ and  $\calE (X)$ is the set of homotopy classes of self-homotopy equivalences of a space $X$ (see Definition \ref{def:selfcloseness}).   

We studied self-closeness numbers of spaces in cofibrations  \cite{OdaYama17} and fibrations \cite{OdaYama18}. This paper is a continuation of \cite{OdaYama17} and the following topics are discussed and some interesting finite cell complexes are studied as examples.

In Section \ref{sec:OdaYamaInprove}  
we improve and reformulate Theorems 3 and 4 of \cite{OdaYama17}   by making use of the homology dimension  $H_*\mbox{-}\dim (X)$ of a space $X$ in place of  $\dim (X)$. 
For any $1$-connected CW-complexes $X$  we show by Theorem \ref{thm:NEXleqHdimX}:   
  {\it $(1)$ $N\calE (X) \leq H_*\mbox{-}\dim (X) + 1$}, and moreover,     {\it $(2)$ If $H_{n} (X)$  is a finitely generated abelian group for  $n=H_{*}\mbox{-}\dim (X)$, then \  $N\calE (X) \leq H_*\mbox{-}\dim (X)$.} 
Notice that  if $X$ is not $1$-connected  it does not hold.
For example,  $H_*\mbox{-}\dim (\R P^{2n}) = 2n-1$   and   \ $N\calE (\R P^{2n}) = 2n$ by Theorem \ref{thm:NERPn}.  We note that $\R P^{2n}$ is not orientable.   These results are interesting comparing with the result of Theoem \ref{thm:XneqiSn}. 
Moreover, Theorem \ref{thm:17Thm3Imporved} shows:     
{\it Let $m \geq 3$ and $B$  a $1$-connected CW-complex with   $H_{*}\mbox{-}\dim (B) \leq m-2$. Assume that $H_{m-2} (B)$  is a finitely generated abelian group.  Let $\gamma  :S^{m-1} \to B$ be a generator   of a direct summand $\Z \subset \pi_{m-1} (B)$.  
If  $a$ is a non-zero integer, then 
 $N\calE (B \cup_{a\gamma} e^{m})  \leq    N\calE (B)$.  } 
We obtain a condition for $N\calE (B \cup_{a\gamma} e^{m}) = N\calE (B)$ by Theorem \ref{thm:17Thm4Improved}.   

In Section \ref{sec:manifolds}   
we study manifolds which satisfy Poincar\'e duality. 
The relation $N{\mathcal E}(X)\leq \dim X$ holds for any $0$-connected finite complex $X$  by Theorem 2 of \cite{ChoiLee15}. Let $[r]$ be the integer part of a rational number $r$. Making use of Theorem \ref{thm:17Thm3Imporved},  we show the following result by Theorem \ref{thm:XneqiSn} which gives  a relation  of self-closeness numbers and  the connectivity for manifolds satisfying Poincar\'{e} duality:  {\it  Let $n \geq 3$. 
Let $X$ be a  connected   closed  $n$-manifold.  If $X$ is $s$-connected for some $1 \leq s \leq [n/2]$ and  $X\not\simeq S^n$, then  $N{\mathcal E}(X)\leq n - s -1$.}

In Section \ref{sec:RProjAndMT} we determine self-closeness numbers of some complexes which are not determined in our previous papers.  We first determine the self-closeness number $N\calE(\R  P^{n})$ for the real projective space $\R P^{n}$ by constructing self-maps so that we have  $N\calE(\R  P^{n}) =  n$ for any integer $n \geq 1$ by Theorem \ref{thm:NERPn}; In \cite{OdaYama17} we studied the case where spaces are $1$-connected and $\R  P^{n}$ was not studied there.

The self-closeness number of the Lens space $L(m;q_{1}, q_{2}, \cdots ,q_{n})$ is determined by Theorem \ref{thm:NELensSpace} as  $N\calE(L(m;q_{1}, q_{2}, \cdots ,q_{n})) = 2n+1$ for  $m \geqq 2$ and $n \geqq 1$.

Concerning the cell complexes defined by  Mimura and Toda in \cite[Section 4]{MimuraToda71},  we study a related complex $(S^{3} \vee \C P^{2})\cup_{\beta} e^{12}$ for a map 
\[\beta =a[[[w, \iota_3],w], \iota_3] + b[[\iota_3,h],h] + c[[w,h],w]  \colon S^{11}  \to S^{3} \vee \C P^{2}  \]   
for some {\it integers} $a,b,c$, where $w=[\iota_3,\iota_2]\in \pi_4(B)$ and $h = h_{5} :S^5\to \C P^2$ is the Hopf map.  Then,   we show that $N\calE (B\cup_{\beta} e^{12})=3$ if and only if  at least one of $a,b,c$ is non-zero (Theorem \ref{thm:MimuraToda}).

In section \ref{sec:RationalHomotopy} we investigate the self-closeness numbers in the category of rationalized $1$-connected finite CW-complexes. In our papers \cite{OdaYama17,OdaYama18} we already studied some properties of the self-closeness numbers making use of the Sullivan and Quillen models \cite{S,Q}; however, in this section we analyze more precise results which  were discussed in the previous papers. 

In the rational homotopy theory, the {\it formality} (see Definition 6.1) is an important property (\cite{DGMS,FHT,NM}).
Let $X_0$ be the rationalization (\cite{HMR}) of a space $X$. 
We first show that (Theorem \ref{thm:HisoXYXformalNE}) 
{\it if $H^*(X;\Q )\cong H^*(Y;\Q)$ as graded algebras and  $X$ is formal,  then  $N{\mathcal E}(X_0)\geq N{\mathcal E}(Y_0)$. }

Let $m \geq 3$ and $B$  a $1$-connected CW-complex with    $H_{*}\mbox{-}\dim (B) \leq m-2$. Let $\gamma  :S^{m-1} \to B$ be a generator   of a direct summand $\Z \subset \pi_{m-1} (B)$. 
Then $N\mathcal{E}( (B\cup_{\gamma} e^n)_0)=n$  if and only if $(B\cup_{\gamma} e^n)_0 \sim (B\vee S^n)_0$  by Proposition \ref{ee}.

A relation between Proposition \ref{ee} and formality is given by Theorem \ref{formalen}  which gives a relation of self-closeness numbers and the formality property \cite{DGMS}: 
{\it Let $n \geq 3$ and $B$  a $1$-connected CW-complex with    $H_{*}\mbox{-}\dim (B) \leq n-2$. Let $\gamma  :S^{n-1} \to B$ be a generator   of a direct summand $\Z \subset \pi_{n-1} (B)$.  Let $B$ be a formal space with $2H_{*}\mbox{-}\dim (B)  <n$.
Then 
 $X=B\cup_{\gamma} e^n$ is formal  if and only if $N\mathcal{E}( X_0)=n$.  } 
As we see in Remark \ref{formal} the condition of Theorem \ref{formalen} is best possible. 
For $1$-connected $n$-manifolds $X$ and $Y$, the connected sum $X\sharp Y$ of $X$ and $Y$ \cite[p. 108]{FOT}  is also an $n$-manifold. 
We prove an inequality  {\it  $N{\mathcal E}((X\sharp Y)_0)\leq N{\mathcal E}((X\vee Y)_0)$} for $1$-connected  compact closed orientable $n$-manifolds  $X$ and $Y$ (Theorem \ref{thm:NEXsharpYleqvee}).  
We also show that 
{\it if $X$ is a $1$-connected   c-symplectic manifold,  then  $N{\mathcal E}(X_0)= 2$} (Theorem  \ref{thm:Xsymplecticmfd}).   
We give  examples  of four cell complexes which  have the same homology groups and four different self-clossness numbers (Example \ref{Exa:fourCellcpx}).  
Four attaching maps of a cell are considered in Example  \ref{Exa:fourAttachingMaps}.    
We also give examples which satisfies $N\calE (X_0) <N\calE (X) $ and examples which satisfy $N\calE (X_0) =N\calE (X) $ (Remark \ref{rem:exaNEXnonRatandRat}). The authors do not know whether or not the relation $N\calE (X_0) \leq N\calE (X)$ holds in general.

Finally, in Section \ref{sec:NstEX} 
we define self-closeness numbers $N_{*}\calE (X)$ and $N^{*}\calE (X)$ of a space $X$ making use of the homology group $H_{i} (X)$ and the cohomology group $H^{i} (X)$ in place of the homotopy group  $\pi_{i} (X)$, respectively.  
If  $X$ and $Y$ have the same homotopy type, then the equalities $N_{*}\calE (X) = N_{*}\calE (Y)$ and $N^{*}\calE (X) = N^{*}\calE (Y)$ hold  (Proposition \ref{prop:NstEXHomInv}). We prove various relations among $N\calE (X)$,  $N_{*}\calE (X)$ and $N^{*}\calE (X)$.

\section{The inequality $N\calE (X) \leq H_*\mbox{-}\dim (X)$} 
\label{sec:OdaYamaInprove} 

In this paper, we consider based spaces and based maps and based homotopies. The symbol  $f \simeq g : X \to Y$ is a based homotopy relation and $[f] : X \to Y$ is the homotopy class of $f$ and $[X,Y]$ is the set of homotopy classes $[f] : X \to Y$.

We define the {\it homology dimension} $H_{*}\mbox{-}\dim (X)$ of a space $X$ by 
\[ H_{*}\mbox{-}\dim (X) :=\max \{\, i \geq 0  \mid H_{i}(X)\neq 0\}.\] 
We understand that $H_{*}\mbox{-}\dim (X) = \infty$ if  $H_{i}(X)\neq 0$ for infinitely many $i$.  Then, making use of the homology dimension, we can improve  Proposition 2 of  \cite{OdaYama17} as  follows: 

\begin{proposition}  
\label{prop:17Prop2Improved} 
Let $m \geq 3$ and $B$ a $1$-connected CW-complex with  $H_{*}\mbox{-}\dim (B) \leq m-2$.   
Let $\gamma =[r] : S^{m-1} \to B$ be a map and $s$  an integer. 
Consider the following diagram\/$:$ 
$$\xymatrix{
S^{m-1}   \ar[rr]^-{\gamma} 
&  & B \ar[d]^{ \ g}  \ar[rr]^-{i}
&  &  B \cup_{\gamma} e^{m} \ar[d]^{\  f  }^{} \ar[rr]^-{q} & & S^{m} \ar[d]^{\  s\iota_{m} } \\ 
S^{m-1} \ar[rr]^-{\gamma}&  &  B   \ar[rr]^-{i}
&  &    B \cup_{\gamma}  e^{m}\ar[rr]^-{q} & & S^{m}  
}$$ 
If the two squares in the above diagram are homotopy commutative, then the following diagram is homotopy commutative\/$:$  
$$\xymatrix{
S^{m-1} \ar[d]_{s\iota_{m-1} \  } \ar[rr]^-{\gamma} 
&  & B \ar[d]^{ \ g}      \\ 
S^{m-1} \ar[rr]^-{\gamma}&  &  B   
}$$ 
It follows that if $\gamma : S^{m-1} \to B$ is a generator of $\pi_{m-1} (B)$ and $g \in \calE(B)$, then $s \gamma$ must be a generator of $\pi_{m-1} (B)$.  
\end{proposition}
\begin{proof} The condition $H_{*}\mbox{-}\dim (B) \leq m-2$ implies $H_{m}  (B \cup_{\gamma} e^{m})\cong \Z$ by the long homology exact sequence for the cofibration. Then the induced homomorphism $f_{*} : H_{m}  (B \cup_{\gamma} e^{m})\cong \Z \to H_{m}  (B \cup_{\gamma} e^{m})\cong \Z$ is the multiple by an integer $s$, and hence,  making use of the Hurewicz isomorphism, the induced homomorphism $f_{\sharp} : \pi_{m} (B \cup_{\gamma} e^{m} , B) \to \pi_{m} (B \cup_{\gamma} e^{m} , B) $ is the multiple by $s$. 

 Let $V^{m}$ be the $m$-disc so that $V^{m} \supset S^{m-1}$. 
Let $\overline{r} : (V^{m}, S^{m-1}) \to (B \cup_{\gamma} e^{m} , B)$ be the characteristic map of the $m$-cell. 
We may assume that $g= f|_B : B \to B$, the restriction of $f: B \cup_{\gamma} e^{m} \to B \cup_{\gamma} e^{m}$ by the given diagram, and therefore, we may consider the following commutative diagram. 
$$\xymatrix{
 \pi_{m} (S^{m}) \ar[d]_{{\rm id}_{\sharp} \ }^{=} & &      \ar[ll]_-{p_{\sharp}}^-{\cong} \pi_{m} (V^{m} , S^{m-1}) \ar[d]_{\overline{r}_{\sharp} \ } \ar[rr]^-{\partial}_-{\cong}       & &   \pi_{m-1} (S^{m-1}) \ar[d]^{r_{\sharp} \ }^{}  \\  
 \pi_{m} (S^{m}) \ar[d]_{s \iota_{m\sharp} \ }^{}  & & \ar[ll]_-{p_{\sharp}}^-{\cong} \pi_{m} (B \cup_{\gamma} e^{m} , B) \ar[d]_{f_{\sharp} \ } \ar[rr]^-{\partial} 
&  &  \pi_{m-1} (B) \ar[d]^{ \ g_{\sharp} }  \\ 
 \pi_{m} (S^{m})   & & \ar[ll]_-{p_{\sharp}}^-{\cong} \pi_{m} (B \cup_{\gamma} e^{m} , B)  \ar[rr]^-{\partial} 
&  &  \pi_{m-1} (B)  
}$$
The class $[\overline{r}]$ is a generator of $\pi_{m} (B \cup_{\gamma} e^{m} , B)$ and $\partial ([\overline{r}])  = [r]$. Therefore we have 
$f_{\sharp} ([\overline{r}]) = s [\overline{r}]$ and hence 
$$g \circ \gamma = [ g \circ r] =   g_{\sharp}  ([r])  =  g_{\sharp}  (\partial ([\overline{r}]))  = \partial (f_{\sharp} ([\overline{r}])) =  \partial ( s [\overline{r}])) = s \partial (  [\overline{r}])) = s [r]   = \gamma \circ s\iota_{m-1} . $$ 
\end{proof}

Choi and Lee \cite{ChoiLee15} introduced the following concept: 
\begin{definition}  \label{def:selfcloseness}  \rm  For any $0$-connected space $X$,  the subset $\calA_{\sharp}^{k} (X)$ of $[X,X]$ is defined by  
$$\calA_{\sharp}^{k} (X)  = \{\, f \in [X,X] \ | \  f_{\sharp} : \pi_{i} (X) \xrightarrow{\cong} \pi_{i} (X) \ \ \textrm{is an isomorphism for any} \ \ i \leq  k\},$$  
and the {\it self-closeness number $N\calE (X)$ of  $X$}  by  
$$N\calE (X)  = \min \{\, k \ | \ \calA_{\sharp}^{k} (X) = \calE (X) \}.$$ 
\end{definition} 

If $\calA_{\sharp}^{k} (X) \neq \calE (X)$ for any $k = 0,\, 1,\, 2,\, 3,\, \cdots$ or $\infty$, then we write $N\calE (X)= \O$. See Example 1 of \cite{ChoiLee15}.

Let \ 
$\pi_{*}\mbox{-}\dim (X):= \max \{ \, i \geq 0 \mid \pi_i(X)\neq 0\}$  \  be the homotopy dimension of $X$. We understand that $\pi_{*}\mbox{-}\dim (X) = \infty$  if $\pi_i(X)\neq 0$ for infinitely many $i$. 
If $X$ is a  connected  CW-complex, then  \ $N\calE (X) \leq \pi_*\mbox{-}\dim (X)$ by the definition of  $N\calE (X)$ and the Whitehead theorem.  We have the following relation between $N\calE (X)$  and $H_{*}\mbox{-}\dim (X)$.

\begin{theorem} \label{thm:NEXleqHdimX} 
  Let $X$ be a $1$-connected CW-complex.  Then, 
\begin{itemize}
\item[{\rm (1)}]  \    $N\calE (X) \leq H_*\mbox{-}\dim (X) + 1$.  
\item[{\rm (2)}] \ If $H_{n} (X)$  is a finitely generated abelian group  for  $n=H_{*}\mbox{-}\dim (X)$, then \  $N\calE (X) \leq H_*\mbox{-}\dim (X)$. 
\end{itemize} 
\end{theorem} 
\begin{proof} (1) \ Let $f: X \to X$ be a map such that $f_\sharp : \pi_k (X) \to   \pi_k (X)$ is an isomorphism for $k \leq  H_{*}\mbox{-}\dim (X) +1$. Then by the Whitehead theorem we have $f_* : H_k (X) \to   H_k (X)$ is an isomorphism for $k \leq H_{*}\mbox{-}\dim (X)$  and  hence  $f_{*} : H_k (X) \to   H_k (X)$ is a bijection for all $k \geq 0$.  It follows that $f \in \calE(X)$. Therefore,  we see $N\calE (X) \leq  H_*\mbox{-}\dim (X) +1$. 

\noindent 
(2) \   Let $f: X \to X$ be a map such that $f_\sharp : \pi_k (X) \to   \pi_k (X)$ is an isomorphism for $k \leq n =H_{*}\mbox{-}\dim (X)$. Then by the Whitehead theorem we have $f_* : H_k (X) \to   H_k (X)$ is an isomorphism for $k \leq n -1$  and $f_{*} : H_n (X) \to   H_n (X)$ is a surjection and hence a bijection by the condition that $H_{n} (X)$   is a finitely generated abelian group. 
It follows that $f_{*} : H_k (X) \to   H_k (X)$ is a bijection for all $k \geq 0$ and hence $f \in \calE(X)$. Therefore,  we see $N\calE (X) \leq n = H_*\mbox{-}\dim (X)$. 
\end{proof}

\begin{example} \rm  If $X$ is not $1$-connected in Theroem  \ref{thm:NEXleqHdimX}, then the result does not hold: 
 Let $X = \R P^{2n}$ for $n \geq 1$. Then $H_*\mbox{-}\dim (\R P^{2n}) = 2n-1$ and $N\calE (\R P^{2n}) = 2n$ by Theorem  \ref{thm:NERPn}.  
\end{example}

Now, consider the following cofibration sequence: 
$$S^{m-1}   \xrightarrow{ \ \  a\gamma  \ \  }   B \xrightarrow{ \ \  i \ \  }  B \cup_{a\gamma} e^{m} \xrightarrow{ \ \  q \ \  }  S^{m}  \xrightarrow{ \ \  \Sigma  a\gamma \ \  }  \Sigma B,$$ 
where $\gamma  :S^{m-1} \to B$ is {\it a generator of a direct summand}\/  $\Z$ of the homotopy group $\pi_{m-1} (B)$ and $a$ is a non-zero integer. 
Theorem \ref{thm:NEXleqHdimX} enables us to improve Theorem 3 of  \cite{OdaYama17} making use of the homology dimension as follows:

\begin{theorem} \label{thm:17Thm3Imporved}   \ 
Let $m \geq 3$ and $B$  a $1$-connected CW-complex with   $H_{*}\mbox{-}\dim (B) \leq m-2$. Assume that $H_{m-2} (B)$  is a finitely generated abelian group.  Let $\gamma  :S^{m-1} \to B$ be a generator   of a direct summand $\Z \subset \pi_{m-1} (B)$.  
If  $a$ is a non-zero integer, then 
 $N\calE (B \cup_{a\gamma} e^{m})  \leq    N\calE (B)$. 
\end{theorem}
\begin{proof}  
We first show that if a self-map $f: B \cup_{a\gamma} e^{m} \to B \cup_{a\gamma} e^{m}$ satisfies  $f \circ i \simeq i \circ g$ \ for a map $g : B \to B$ such that $g   \in \calE(B)$, then  $f \in \calE(B \cup_{a\gamma} e^{m})$. Consider the following diagram: 
$$\xymatrix{
S^{m-1} \ar@{.>}[d]_{s\iota_{m-1}} \ar[rr]^-{a\gamma} 
&  & B \ar[d]^{ \ g}  \ar[rr]^-{i}
&  &  B \cup_{a\gamma} e^{m} \ar[d]^{f \ } \ar[rr]^-{q} & & S^{m} \ar[d]_{s\iota_{m} \ } \ar[rr]^-{\Sigma a\gamma} 
&  & \Sigma B \ar[d]^{ \ \Sigma g}  \\ 
S^{m-1} \ar[rr]^-{a\gamma}&  &  B   \ar[rr]^-{i}
&  &    B \cup_{a\gamma}  e^{m}\ar[rr]^-{q} & & S^{m} \ar[rr]^-{\Sigma a\gamma} 
&  & \Sigma B 
}$$ 
Since $g \in \calE(B)$,  there exists a map $\overline{g} : B \to B$ such that 
$\overline{g} \circ g  \simeq  1_{B} \simeq  g \circ \overline{g}$.   
There exists an integer $s$ such that $q \circ f \simeq (s \iota_m ) \circ q$.  It follows that 
$g \circ a\gamma \simeq a\gamma \circ s \iota_{m-1}  \simeq  sa \gamma$ 
by Proposition \ref{prop:17Prop2Improved} since  $H_{*}\mbox{-}\dim (B) \leq m-2$.   Then we see 
$$a\gamma \simeq \overline{g} \circ g \circ a\gamma \simeq \overline{g} \circ  s a\gamma \simeq  s a (\overline{g} \circ  \gamma).$$ 
Let us write $\pi_{m-1} (B) \cong \Z \{\gamma\} \oplus G$ (direct sum decomposition) for some subgroup $G$ by assumption. Then we may write $\overline{g} \circ \gamma = t \gamma + x$ for some integer $t$ and some $x \in G$. It follows that 
$$a \gamma = sa (t \gamma + x) = sa t \gamma + sax$$ 
or $a = sat$ in $\Z$ and $sax = 0$. 
It follows that $s= \pm 1$ since $\gamma$ is a generator of a direct summand $\Z$, and hence we see that the induced homomorphism 
$$f^{*} : H^{d}(B \cup_{a\gamma} e^{m})  \to H^{d}(B \cup_{a\gamma} e^{m})$$ 
is an isomorphism  for any $d \geq 0$ by the five lemma. It follows that $f$ is a homotopy equivalence.

Finally we prove that $N\calE (B \cup_{a\gamma} e^{m}) \leq  N\calE (B)$. 
Consider the following homology exact sequence: 
$$\cdots \to  H_{d+1} (B \cup_{a\gamma} e^{m} , B) \xrightarrow{ \ \partial  \ }  H_{d} (B) \xrightarrow{ \ i_{*}  \ } H_{d} (B \cup_{a\gamma} e^{m}) \xrightarrow{ \ j_{*} \ } H_{d} (B \cup_{a\gamma} e^{m} , B) \to \cdots.$$ 
Since  $H_{d+1} (B \cup_{a\gamma} e^{m} , B)  = 0$ for $d+1 \leq m-1$ and $H_{d} (B \cup_{a\gamma} e^{m} , B)  = 0$ for $d  \leq m-1$, the induced homomorphism  $i_{*} : H_{d} (B) \to H_{d} (B \cup_{a\gamma} e^{m})$  is an isomorphism for $d  \leq m-2$ and  $i_{*} : H_{m-1} (B) \to H_{m-1} (B \cup_{a\gamma} e^{m})$ is an epimorphism. It follows that  $i_{\sharp} : \pi_{d} (B) \to \pi_{d} (B \cup_{a\gamma} e^{m})$ is an isomorphism for $d  \leq m-2$ and  $i_{\sharp} : \pi_{m-1} (B) \to \pi_{m-1} (B \cup_{a\gamma} e^{m})$ is an epimorphism by the Whitehead theorem.   

We know $N\calE (B) \leq H_*\mbox{-}\dim (B) \leq m-2$ by Theorem \ref{thm:NEXleqHdimX}(2). Therefore, assume that $N\calE (B) = k$  for some natural number $k \leq m-2$. Suppose that a map $f :B \cup_{a\gamma} e^{m} \to B \cup_{a\gamma} e^{m}$  satisfies $f_{\sharp} : \pi_{d} (B \cup_{a\gamma} e^{m}) \to \pi_{d} (B \cup_{a\gamma} e^{m})$ is an isomorphism for any $1 \leq  d \leq k$.  

 We may assume that $\dim (B) \leq m-1$ by Proposition 4C.1 (p.429) of Hatcher \cite{Hatcher01}, since $H_{*}\mbox{-}\dim (B) \leq m-2$. Therefore,  we have a map $g= f|_B : B \to B$, the restriction of $f: B \cup_{\gamma} e^{m} \to B \cup_{\gamma} e^{m}$, by the cellular approximation theorem. Then we have the following  homotopy commutative diagram: 
$$\xymatrix{
  B \ar[d]_{g} \ar[rr]^-{i}
&  &  B \cup_{a\gamma} e^{m} \ar[d]^{ \  f  }^{}    \\ 
 B   \ar[rr]^-{i}
&  &    B \cup_{a\gamma}  e^{m} 
}$$
which implies the following commutative diagram for any $d \geq 1$: 
$$\xymatrix{
\pi_{d} (B) \ar[d]_{g_{\sharp} \  } \ar[rr]^-{i_{\sharp}}
&  &  \pi_{d} (B \cup_{a\gamma} e^{m}) \ar[d]^{ \ f_{\sharp}  }^{}    \\ 
 \pi_{d} (B)   \ar[rr]^-{i_{\sharp}}
&  &    \pi_{d} (B \cup_{a\gamma}  e^{m}) 
}$$
Then $g_{\sharp} : \pi_{d} (B) \to \pi_{d} (B)$ is an isomorphism for any $1 \leq d \leq k$  and hence $g \in \calE (B)$. 
It follows that $f \in \calE (B  \cup_{a\gamma}  e^{m})$ by the first part of this proof. Hence we have $N\calE (B \cup_{a\gamma} e^{m}) \leq  k = N\calE (B)$.   
\end{proof}

The following result is a reformulation of Theorem 4 of \cite{OdaYama17} making use of $H_{*}\mbox{-}\dim (B)$ in place of $\dim (B)$.

\begin{theorem}   \label{thm:17Thm4Improved}  \  
Let $m \geq 3$.  Let $B$ be a $1$-connected CW-complex with   $H_{*}\mbox{-}\dim (B) \leq m-2$. Assume that $H_{m-2} (B)$  is a finitely generated abelian group.  If $\pi_{m-1} (B) \cong \Z$ and $\gamma  :S^{m-1} \to B$ is a generator   of $\pi_{m-1} (B)$ and $a$ is a non-zero integer, 
then $N\calE (B \cup_{a\gamma} e^{m}) = N\calE (B)$. 
\end{theorem}
\begin{proof} We know $N\calE (B \cup_{a\gamma} e^{m}) \leq   N\calE (B)$ by Theorem \ref{thm:17Thm3Imporved}. Therefore, we show that 
\[ N\calE (B \cup_{a\gamma} e^{m}) \geq   N\calE (B) \] in the following argument.  Let  $N\calE (B \cup_{a\gamma} e^{m})=k$ for some integer $k$. We may assume that  $1 \leq k \leq m-2$, since $N\calE (B) \leq  H_{*}\mbox{-}\dim (B) \leq m-2$ by Theorem \ref{thm:NEXleqHdimX}. 

Let $g: B \to B$ be a map such that $g_{\sharp} : \pi_{d} (B) \to \pi_{d} (B)$ is an isomorphism for any $1 \leq  d \leq k$.  
Consider the following diagram: 
$$\xymatrix{
S^{m-1} \ar@{.>}[d]_{s\iota_{m-1}} \ar[rr]^-{a\gamma} 
&  & B \ar[d]^{ \ g}  \ar[rr]^-{i}
&  &  B \cup_{a\gamma} e^{m} \ar@{.>}[d]^{h \ }^{} \ar[rr]^-{q} & & S^{m} \ar@{.>}[d]_{s\iota_{m} \ } \ar[rr]^-{\Sigma a\gamma} 
&  & \Sigma B \ar[d]^{ \ \Sigma g}  \\ 
S^{m-1} \ar[rr]^-{a\gamma}&  &  B   \ar[rr]^-{i}
&  &    B \cup_{a\gamma}  e^{m}\ar[rr]^-{q} & & S^{m} \ar[rr]^-{\Sigma a\gamma} 
&  & \Sigma B 
}$$ 
Since $\pi_{m-1} (B) \cong \Z$,  there exists an integer  $s$ such that 
$g \circ a\gamma \simeq sa \gamma \simeq a\gamma \circ s \iota_{m-1}$. 
Then  we have a self-map $h: B \cup_{a\gamma} e^{m} \to B \cup_{a\gamma} e^{m}$ which makes the above diagram homotopy commutative.   
Consider again the following commutative diagram for any $d \geq 1$: 
$$\xymatrix{
\pi_{d} (B) \ar[d]_{g_{\sharp} \  } \ar[rr]^-{i_{\sharp}}
&  &  \pi_{d} (B \cup_{a\gamma} e^{m}) \ar[d]^{ \ h_{\sharp}  }^{}    \\ 
 \pi_{d} (B)   \ar[rr]^-{i_{\sharp}}
&  &    \pi_{d} (B \cup_{a\gamma}  e^{m}) 
}$$
Since $i_{\sharp} : \pi_{d} (B) \to \pi_{d} (B \cup_{a\gamma} e^{m})$ is an isomorphism  for $d  \leq m-2$ and $g_{\sharp} : \pi_{d} (B) \to \pi_{d} (B)$ is an isomorphism for any $1 \leq  d \leq k$, we see $h_{\sharp} : \pi_{d} (B \cup_{a\gamma}  e^{m}) \to \pi_{d} (B \cup_{a\gamma}  e^{m})$ is an isomorphism for any $1 \leq  d \leq k$. Hence $h \in \calE(B \cup_{a\gamma} e^{m})$ by our assumption $N\calE (B \cup_{a\gamma} e^{m})=k$.  
Therefore $g_{\sharp} : \pi_{d} (B) \to \pi_{d} (B)$ is an isomorphism for any $1 \leq d \leq m-2$  and hence $g \in \calE (B)$  by Theorem \ref{thm:NEXleqHdimX}(2). It follows that  $N\calE (B) \leq  k =  N\calE (B \cup_{a\gamma} e^{m})$.
\end{proof}

\section{Closed manifolds} \label{sec:manifolds} 

A connected compact topological $n$-manifold without boundary is called a {\it closed} $n$-manifold. An oriented closed $n$-manifold enjoys Poincar\'e Duality. Any $1$-connected manifold is orientable (\cite[Proposition 3.25]{Hatcher01}).

\begin{proposition} \label{prop:HstdimB} 
Let $X$ be a $1$-connected closed $n$-manifold. 
Let $x$ be a point of $X$. Then the following results hold\/$:$
\[H_{d} (X-\{x\}) \cong \left\{ 
\begin{array}{@{\,}ll} 
   0 &  \mbox{if} \ \ d \geq n-1      \\ 
H_d (X) &   \mbox{if} \ \ 0 \leq  d \leq n -2   
\end{array} 
\right.\]   

\end{proposition} 
\begin{proof} For the sequence 
\[X-\{x\} \xrightarrow{\ i_{x} \ } X \xrightarrow{\ j_{x} \ } (X, X-\{x\}),\]  
we have the long homology exact sequence of the pair $(X, X-\{x\})$. In the exact sequence we see $H_1 (X) = 0$ since $X$ is $1$-connected, and hence we have $H_{n-1} (X)= 0$ by the Poincar\'e duality (see 6.18 Poincar\'e Duality Theorem of Vick \cite{Vick95}). 
We see $H_n (X-\{x\}) = 0$ and the induced homomorphism $j_{x*} : H_n (X)\cong \Z \to H_n (X, X-\{x\})\cong \Z$ is an isomorphism (see 6.5 Theorem and 6.7 Corollary of Vick \cite{Vick95}).  
It follows that 
\[H_{d} (X-\{x\}) \cong \left\{ 
\begin{array}{@{\,}ll} 
   0 &  \mbox{if} \ \ d \geq n-1      \\ 
H_d (X) &   \mbox{if} \ \ 0 \leq  d \leq n -2   
\end{array} 
\right.\] 
\end{proof}

\begin{proposition} \label{prop:XsimeqBcupgamen}
Let $X$ be a  $1$-connected   closed   $n$-manifold.
Then $X \simeq B  \cup_{\alpha} e^{n}$ for some  subcomplex $B$ such that $B \simeq X-\{x\}$ for a point $x$ of $X$ with $\dim (B) < n$ and an attaching map $\alpha : S^{n-1} \to B$.
\end{proposition}
\begin{proof} We choose a neighborhood ${\rm Int} (D^{n})$ of $x$ so that   
\[ x \in {\rm Int} (D^{n}) = e^{n} \subset X = B  \cup_{\alpha} e^{n} ,\]  
where $B = X \setminus {\rm Int} (D^{n})$  and $\alpha : S^{n-1} \to D^{n} \setminus {\rm Int} (D^{n}) \subset B$ is the inclusion.  Then $B \simeq X-\{x\}$.  Therefore, we have the result by Proposition  \ref{prop:HstdimB} and Proposition 4C.1 (p.429) of Hatcher \cite{Hatcher01}. 
\end{proof}

\begin{remark} \rm 
If $X$ is not $1$-connected, then the following examples are interesting. 
\begin{itemize}
\item[{\rm (1)}] \  Let $X = (S^{1})^{\times n} \ (n \geq 2)$. Then $\pi_{1} (X) \cong H_{1} (X) \cong H^{1} (X) \cong  \Z^{n}$. It follows that $H^{n-1} (X) \cong  H_{1} (X) \cong  \Z^{n}$ and  $H^{1} (X) \cong  H_{n-1} (X) \cong  \Z^{n}$. 
\item[{\rm (2)}] \  The real projective space $\R P^{2k+1} \ (k \geq 1)$ is orientable. Let $X = (\R P^{2k+1})^{\times n}$. Then, $\pi_{1} (X) \cong H_{1} (X) \cong (\Z /2)^{n}$. It follows that $H^{1} (X) \cong  H_{n(2k+1)-1} (X) \cong  0$.
\end{itemize} 
\end{remark}

\begin{lemma} \label{lem:orderOfgamma} 
Let $n \geq 3$. Let $X$ be a  $1$-connected   closed  $n$-manifold.   
If $X=B\cup_{\alpha} e^n$ with  $B_{0}\not\simeq *$, then the  attaching map $\alpha : S^{n-1} \to B$ is of infinite order. 
Therefore, there exists a generator $\gamma$ of a direct summand $\Z \subset \pi_{n-1} (B)$ such that $\alpha = a \gamma$ for some non-zero integer $a$.  
\end{lemma}
\begin{proof} If $\alpha$ is of finite order, then we see the rationalization $X_0$ of $X$ decomposes as $X_0 \simeq B_0 \vee S^{n}_0$ (see  Proof of Proposition 1.3 (p.55) of \cite{HMR}). 
This contradicts the fact that  there exists an isomorphism 
\[D : H^{p} (X;\Q) \cong H_{n-p} (X;\Q) \] 
given by $D(x) = \mu \cap x$ for any $x \in H^{p} (X;\Q)$ by the Poincar\'e duality, where $\mu \in H_{n} (X;\Q)$ is the fundamental class (see 6.18 Poincar\'e Duality Theorem and 6.30 Proposition  of Vick \cite{Vick95}), since   for any non-zero element $x \in H^p (X;\Q)$, there exists an element  $y \in H^{n-p} (X;\Q)$ such that $\langle \mu \cap x,  y \rangle \neq 0$ and we have a formula: 
\[ 0 \neq \langle \mu \cap x,  y \rangle  =  \langle \mu , x \cup y \rangle    = 0 . \] 
Here, we note that  $x \cup y = 0$ since $X_0 \simeq B_0 \vee S^{n}_0$ and 
\[ \langle \mu, z \rangle  = \varepsilon_* (\mu \cap z) \ \ (\mu \in H_n (X;\Q) , \ \  z \in H^n (X;\Q), \] 
where $\varepsilon_*  : H_0 (X;\Q) \to H_0 (P;\Q) = \Q$ is induced by the map $\varepsilon :  X \to P$, where $P$ is a point.    

\end{proof}

Let $[r]$ be the integer part of a rational number $r$.

\begin{theorem} \label{thm:XneqiSn} 
Let $n \geq 3$ and $1 \leq s \leq [n/2]$ be integers.   If $X$ is $s$-connected  closed $n$-manifold  and  $X_{0} \not\simeq S^{n}_{0}$, then  $N{\mathcal E}(X)\leq n - s -1$.
\end{theorem}
\begin{proof} We have  $X=B\cup_{a\gamma} e^n$ with  $B_{0}\not\simeq *$,  
where $\gamma : S^{n-1} \to B$ is a generator of the direct summand of $\Z \subset \pi_{n-1} (B)$ and $a$ is a non-zero integer  by Proposition \ref{prop:XsimeqBcupgamen} and Lemma  \ref{lem:orderOfgamma}.
We see  $H_{n-2} (B) = H_{n-2} (X)$ is an abelian group of finite type, since $X$ has a homotopy type of a finite complex by Theorem 1 of Milnor  \cite{Milnor59} or  Corollary A.12.  (p.529) of Hatcher  \cite{Hatcher01} (see also  Application 1 (p.36) of Milnor  \cite{Milnor63} and Theorem 2  (p.384) of Borisovich et al. \cite{BBFI95} for  differentiable manifolds).  Then,  by   Theorem \ref{thm:17Thm3Imporved}, we have $N{\mathcal E}(X)\leq N{\mathcal E}(B)$. 
Since $H_{*}\mbox{-}\dim (B) =  H_{*}\mbox{-}\dim (X-\{x\})$  for a point $x$ of $X$  by Proposition  \ref{prop:XsimeqBcupgamen}, we have $H_{*}\mbox{-}\dim (B) \leq n-s-1$ by Proposition \ref{prop:HstdimB} and the Poincar\'e duality.  
It follows that $N{\mathcal E}(X) \leq n-s-1 $  by Theorem  \ref{thm:17Thm3Imporved}.
\end{proof}

\begin{proposition}  Let $KB$ be the Klein bottle, $T^2 = S^{1} \times S^{1}$ the torus and  $\R P^2$ the real projective plane. Then  
$N\calE(KB) = 1$,   $N\calE(T^2) = 1$ and $N\calE(\R P^2) = 2$. 

\end{proposition} 
\begin{proof} Since $\pi_{1} (KB) \cong \langle \alpha, \beta \ | \ \alpha \beta  \alpha \beta^{-1} \rangle$,    $\pi_{2} (KB) \cong 0$ by the fibration $S^1 \to KB \to S^1$  and    $\pi_{1} (T^2) \cong \Z \oplus \Z$, $\pi_{2} (T^2) \cong 0$,  we have  $N\calE(KB) = N\calE(T^2) = 1$.   We see   $\pi_{1} (\R P^2) \cong \Z /2$, $\pi_{2} (\R P^2) \cong \Z$ by the fibration $S^0 \to S^2 \to \R P^2$; however we need geometrical consideration to determine  $N\calE(\R P^2) = 2$  as in Theorem  \ref{thm:NERPn}. 
\end{proof} 

\section{Real projective spaces, lens spaces and a cell complex of Mimura and Toda} \label{sec:RProjAndMT}

\subsection{Self-closeness numbers of real projective spaces}  
We can not apply Theorems 3 and 4 of \cite{OdaYama17} to determine $N\calE(\R P^{n})$, since these theorems assume $1$-connectedness for the spaces. Therefore, in this section we determine $N\calE(\R P^{n})$ by constructing special self-maps of $\R P^{n}$.  

As is written on p.232 of Munkres \cite{Munkres84}, the $m$-sphere $S^m$ is decomposed by the cells 
$$e_{+}^{n}, \ e_{-}^{n} \  \  (n = 0, 1, 2, \cdots , m).$$   
The characteristic map $q_{\pm} : (V^n , S^{n-1}) \to (E_{\pm}^n , S^{n-1})$ of $e_{\pm}^{n}$ is defined by 
$$q_{\pm}(t_1 , t_2, \cdots, t_n) = (t_1 , t_2, \cdots, t_n, \pm \sqrt{1 - t_{1}^2 - \cdots - t_{n}^2}). $$

\begin{theorem} \label{thm:NERPn}  \  
$N\calE(\R  P^{n}) =  n $ for any integer $n \geq 1$.    
\end{theorem} 
\begin{proof}  
Let $S^1 = \{ (x, y) \ | \  x^2 + y^2 =1 \, \} \subset \R^2$. Let $(\cos 2\pi t, \sin 2\pi t) \in S^1$ for  $0 \leq t <1$.  Let  $k \in \Z$ be any integer. We define $f_k : S^1 \to S^1$ by 
$$f_k (\cos 2\pi t, \sin 2\pi t) = (\cos 2k \pi t, \sin 2k \pi t).$$ 
Then the following diagram is commutative:  
\[\xymatrix{
   S^{1} \ar[d]_{p_{1} \ } \ar[rr]^-{f_{k}} & &    S^{1} \ar[d]^{\  p_{1} }       \\ 
  \R P^{1} \ar[rr]^-{g_{k}} & &     \R P^{1}   
}\]
where $\R P^{1} \cong S^{1}$ and $p_1  (\cos 2\pi t, \sin 2\pi t) = (\cos 4 \pi t, \sin 4 \pi t)$ and $f_k = g_k$.

Let  $n \geq 2$.  By the construction of the suspension space, we have a map $f_{k}^{n} : S^{n} \to S^{n}$ of degree $k$, which induces the map 
$g_{k}^{n} : \R P^{n} \to \R P^{n}$ for any odd integer $k$ such that the following diagram is commutative:  
\[\xymatrix{
   S^{n} \ar[d]_{p_{n} \ } \ar[rr]^-{f_{k}^{n}} & &    S^{n} \ar[d]^{\  p_{n} }       \\ 
  \R P^{n} \ar[rr]^-{g_{k}^{n}} & &     \R P^{n}   
}\] 
It follows that 
$g_{k\sharp}^{n} (p_n \circ s \iota_{n}) = p_n \circ ks \iota_{n}$.  
We note that the following diagram is homotopy commutative  if $k$ is odd ($n \geq 2)$: 
\[\xymatrix{
   S^{1} \ar[d]_-{i \ } \ar[rr]^-{1_{S^{1}}} & &    S^{1} \ar[d]^-{\  i }       \\ 
  \R P^{2}  \ar[d]_-{i \  }  \ar[rr]^-{g_{k}^{2}} & &     \R P^{2}  \ar[d]^-{\  i }  \\ 
  \R P^{n} \ar[rr]^-{g_{k}^{n}} & &     \R P^{n}   
}\] 
where $i: S^{1} =  \R P^{1} \to \R P^{2}$ is the inclusion map.  
Therefore,  $g_{k\sharp}^{n} : \pi_{1}(\R P^{n}) \to \pi_{1}(\R P^{n})$ is  the identity if $k$ is odd ($n \geq 2)$. 
If $n \geq 2$ and $k$ is odd, then $g_{k}^{n} : \R P^{n} \to \R P^{n}$ induces an isomorphism  $g_{k\sharp}^{n} : \pi_{s}(\R P^{n}) \to \pi_{s}(\R P^{n})$ for any $s <n$, but $g_{k}^{n}$ is not a homotopy equivalence if $k \neq \pm 1$. 
\end{proof}

\subsection{Self-closeness numbers of lens spaces}  

We refer the reader to Section 8 of Olum \cite{Olum53} for the generalized lens space  $L(m;q_{1}, q_{2}, \cdots ,q_{n})$.

Let  $\C$ be the field of complex numbers and $(z_{0}, z_{1}, \cdots, z_{n}) \in \C^{n+1}$  be any element of complex $(n+1)$-dimensional vector spacce.  We write the complex numbers in the polar form as  
\[z_{0} = r_{0} e^{i\alpha_{0}},  \ \  z_{1} = r_{1} e^{i\alpha_{1}}, \ \cdots, \ \  z_{n} = r_{n} e^{i\alpha_{n}}.\] 
Then, the $(2n+1)$-dimensional sphere 
 $S^{2n+1} \subset \C^{n+1}$  is  defined by the formula 
 \[r_{0}^{2} + r_{1}^{2} + \cdots + r_{n}^{2} = 1 .\]   
  
  Let $m \geqq 2$ and $n \geqq 1$ be  fixed integers and let $q_{1}, q_{2}, \cdots ,q_{n}$ be $n$ integers relatively prime to $m$. 
A homeomorphism  $\gamma : S^{2n+1} \to S^{2n+1}$ is defined by 
\[\gamma (z_{0}, z_{1}, \cdots, z_{n}) =  (z_{0}e^{2\pi i / m}, z_{1}e^{2\pi i q_{1}/ m}, \cdots, z_{n}e^{2\pi i q_{n}/ m})\]  
Let $\Gamma =\{1, \gamma , \gamma^{2}, \cdots, \gamma^{m-1}\}$ be the cyclic group  generated by $\gamma$, where we identify 
\[\gamma = (e^{2\pi i / m}, e^{2\pi i q_{1}/ m}, \cdots, e^{2\pi i q_{n}/ m}).\] 
 Then the lens space $L(m;q_{1}, q_{2}, \cdots ,q_{n})$ is defined by the orbit space 
\[ L(m;q_{1}, q_{2}, \cdots ,q_{n}) =  S^{2n+1} / \Gamma\] 
Let $L(m) = S^{1} / \{e^{i 2\pi   /m} \}$ be the orbit space. Then we may consider the subspaces: 
\[ L(m) \subset  L(m;q_{1})  \subset L(m;q_{1}, q_{2})  \subset  \cdots \subset L(m;q_{1}, q_{2}, \cdots ,q_{n-1}) \subset  L(m;q_{1}, q_{2}, \cdots ,q_{n}) . \] 

Let $p_{n} = p_{n\C}:S^{2n+1} \to  L(m;q_{1}, q_{2}, \cdots ,q_{n})$ be the natural projection. We have a fibration $(n \geq 1)$: 
\[ \Z/m   \xrightarrow{ \ \  i_{n}  \   \  }    S^{2n+1}  \xrightarrow{ \ \  p_{n}  \   \  }  L(m;q_{1}, q_{2}, \cdots ,q_{n}).  \] 
We see \ $\pi_{1} (L(m;q_{1}, q_{2}, \cdots ,q_{n})  ) \cong \Z/m$ \ and 
\[\pi_{2n+1} (L(m;q_{1}, q_{2}, \cdots ,q_{n})  ) \cong \Z \{[p_{n}]\} , \] 
where $[p_{n}]$ is the homotopy class of the map \  $p_{n} :S^{2n+1} \to  L(m;q_{1}, q_{2}, \cdots ,q_{n})$.

\begin{theorem} \label{thm:NELensSpace}
 Let $m \geqq 2$ and $n \geqq 1$.  Then 
$N\calE(L(m;q_{1}, q_{2}, \cdots ,q_{n})) = 2n+1$.  
\end{theorem}
\begin{proof}  For any integer $d$,    we define $f_{d} = f_{d}^{n} : S^{2n+1} \to S^{2n+1}$ by 
\[f_{d} (r_{0} e^{i\alpha_{0}},   r_{1} e^{i\alpha_{1}}, \ \cdots,    r_{n} e^{i\alpha_{n}}) = (r_{0} e^{i  \alpha_{0} d},   r_{1} e^{i\alpha_{1}}, \ \cdots,    r_{n} e^{i\alpha_{n}}) .\]
If $d = ms +1$ for some integer $s$, then we have 
\[ f_{d} ( \gamma (r_{0} e^{i\alpha_{0}},   r_{1} e^{i\alpha_{1}}, \ \cdots,    r_{n} e^{i\alpha_{n}}) ) =  \gamma f_{d} (r_{0} e^{i\alpha_{0}},   r_{1} e^{i\alpha_{1}}, \ \cdots,    r_{n} e^{i\alpha_{n}}) . \] 
Therefore, a map $g_{ms+1} = g_{ms+1}^{n} :L(m;q_{1}, q_{2}, \cdots ,q_{n}) \to L(m;q_{1}, q_{2}, \cdots ,q_{n})$ is induced by the map $f_{ms+1} : S^{2n+1} \to S^{2n+1}$  and we have the following homotopy commutative diagram:  
\[\xymatrix{
   S^{2n+1} \ar[d]_-{p_{n} \ } \ar[rr]^-{f_{sm+1}} & &    S^{2n+1} \ar[d]^-{\  p_{n} }       \\ 
 L(m;q_{1}, q_{2}, \cdots ,q_{n}) \ar[rr]^-{g_{sm+1}} & &    L(m;q_{1}, q_{2}, \cdots ,q_{n}) 
}\]

Now, we consider the case $n=0$, namely $L(m)$: Let $S^1 = \{ e^{i 2\pi  t }  \ | \  t \in  \R \} \subset \C$.
Let $ e^{i 2\pi  t }  \in S^1$ for  $0 \leq t <1$ and  $k \in \Z$  any integer. Then $f_{k}^{0} : S^1 \to S^1$ is defined by $f_{k}^{0} (e^{i 2\pi  t  }) = e^{i 2\pi  t k }$ and the following diagram is commutative:  
$$\xymatrix{
   S^{1} \ar[d]_{p_{0} \ } \ar[rr]^-{f_{sm+1}^{0}} & &    S^{1} \ar[d]^{\  p_{0} }       \\ 
 L(m) \ar[rr]^-{g_{sm+1}^{0}} & &     L(m)  
}$$
where $L(m) \cong S^{1}$ and $p_0  (e^{i 2\pi  t }) = e^{i 2\pi  t m}$ and $f_{k}^{0} = g_{k}^{0}$. 
The following  diagram is homotopy commutative: 
\[\xymatrix{
   S^{1}  = L(m) \ar[d]_-{i \ } \ar[rr]^-{1_{S^{1}}} & &    S^{1}  = L(m)\ar[d]^-{\  i }       \\ 
 L(m;q_{1}) \ar[rr]^-{g_{sm+1}^{1}} \ar[d]_-{\cap \ }& &     L(m;q_{1}) \ar[d]^-{\ \cap   }  \\ 
L(m;q_{1}, q_{2}, \cdots ,q_{n}) \ar[rr]^-{g_{sm+1}^{n}} & &    L(m;q_{1}, q_{2}, \cdots ,q_{n})  
}\] 
where $i: S^{1} = L(m) \to L(m;q_{1})$ is the inclusion map   (see p.144 Example 2.43 Lens spaces \cite{Hatcher01}).  It follows that 
\[g_{sm+1 \sharp}^{n} : \pi_{1}( L(m;q_{1}, q_{2}, \cdots ,q_{n})) \to \pi_{1}( L(m;q_{1}, q_{2}, \cdots ,q_{n}))\] 
is the identity homomorphism. 
Let $L = L(m;q_{1}, q_{2}, \cdots ,q_{n})$.  We have the following commutative diagram between the covering space: 
$$\xymatrix{
\Gamma = \Z/m\Z  \ar[d]_{\times (ms+1) \ } \ar[rr]^-{i} & & S^{2n+1} \ar[d]_{f_{ms+1}^{n} \ } \ar[rr]^-{p_{n}} & &  L \ar@{.>}[d]^{ \ g_{ms+1}^{n}  }   \\ 
\Gamma =  \Z/m\Z   \ar[rr]^-{i} & &    S^{2n+1}     \ar[rr]^-{p_{n}} & &    L  & &     
}$$
$$\xymatrix{
  \pi_{1} (L) \ar[d]_{g_{ms+1\sharp}^{n} \ } \ar[rr]^-{\Delta}_-{\cong} & & \Gamma =  \Z/m\Z \ar[d]^{ \ \times (ms+1)  }  &   \\ 
  \pi_{1} (L)    \ar[rr]^-{\Delta}_-{\cong} & & \Gamma=  \Z/m\Z &        
}$$
where $\Delta (\alpha) =  \alpha \cdot e_{0}$ by (3.2.8) Proposition (p.69) of tom Dieck \cite{tomDieck}.

We see that $g_{ms+1}^{n}$ induces the identity homomorphism for $g_{ms+1\sharp}^{n} : \pi_{1} (L) \cong \Z/m\Z  \to \pi_{1} (L)\cong \Z/m\Z $, \  $g_{ms+1\sharp}^{n} : \pi_{k} (L) \cong 0  \to \pi_{k} (L)\cong 0 \  (2 \leq k \leq 2n)$  and the multiplication by $ms + 1$ for the induced homomorphism $g_{ms+1\sharp}^{n} : \pi_{2n+1} (L) \cong \Z\{[p_{n}]\}   \to \pi_{2n+1} (L)\cong \Z \{[p_{n}]\}$.  Therefore, $N\calE(L(m;q_{1}, q_{2}, \cdots ,q_{n}))=2n+1$. 
\end{proof}

\subsection{A cell complex defined by Mimura and Toda} 
To show that there exists a 4-cell complex which is not $p$-universal, Mimura and Toda \cite[Theorem 4.2]{MimuraToda71} constructed a 4-cell complex $K_q = (S^{3} \vee \C P^{2})\cup_{q\alpha} e^{12}$ for any integer $q$ and the sum of the Whitehead products 
\[\alpha =[[[w, \iota_3],w], \iota_3] + [[\iota_3,h],h] + [[w,h],w] : S^{11}  \to S^{3} \vee \C P^{2} \] 
where $w=[\iota_3,\iota_2]\in \pi_4(S^{3} \vee \C P^{2})$ and $h = h_{5}:S^5\to \C P^2$ is the Hopf map. 
The complex we study in this section is a modification of the complex $K_q$ of Mimura and Toda, that is, we study a related complex $(S^{3} \vee \C P^{2})\cup_{\beta} e^{12}$ with a map 
\[\beta =a[[[w, \iota_3],w], \iota_3] + b[[\iota_3,h],h] + c[[w,h],w]  \colon S^{11}  \to S^{3} \vee \C P^{2}  \]   
for some {\it integers} $a,b,c$. Then, the following result is obtained.

\begin{theorem} \label{thm:MimuraToda}   
 Let $B= S^{3} \vee \C P^{2}$ and $\iota_{3} \colon  S^{3} \subset B$ and $\iota_{2} \colon  S^{2}  \subset  B$ be the homotopy classes of the inclusions and $\omega = [\iota_{3}, \iota_{2}] \colon S^{4} \to B$ the Whitehead product. Let $h \colon S^{5} \to \C P^{2} \subset B$ the composite of the Hopf map and the inclusion.  
Consider the following sum of three Whitehead products\/$:$ 
\[\beta =a[[[w, \iota_3],w], \iota_3] + b[[\iota_3,h],h] + c[[w,h],w]  \colon S^{11}  \to B  \]   
for some {\it integers} $a,b,c$. Then,  $N\calE (B\cup_{\beta} e^{12})=3$ if and only if  at least one of $a,b,c$ is non-zero. 
\end{theorem} 
\begin{proof} 
In Section 4 of \cite{MimuraToda71}, it is proved that 
\[ \pi_{2} (B) \cong \Z \{\iota_{2}\}, \   \pi_{3} (B) \cong \Z \{\iota_{3}\}, \   \pi_{4} (B) \cong \Z \{\omega\} \oplus \Z/2  \   \mbox{and}\   \pi_{5} (B) \cong \Z  \{h\} \oplus \Z/2 \oplus \Z/2 .  \]  
Here, $\Z\{\alpha\}$ means that $\alpha$ is a generator of the free part $\Z$. 

Let $f: B\cup_{\beta} e^{12} \to B\cup_{\beta} e^{12}$ be a map. Let $g= f|_{B} : B \to B$ be the restriction of $f$. 
Consider the following diagram obtained by Proposition 2 of \cite{OdaYama17}: 
$$\xymatrix{
S^{11} \ar@{.>}[d]_{s\iota_{11}} \ar[rr]^-{\beta} 
&  & B \ar[d]^{ \ g}  \ar[rr]^-{i}
&  &  B \cup_{\beta} e^{12} \ar[d]^{f \ } \ar[rr]^-{q} & & S^{12} \ar[d]_{s\iota_{12} \ } \ar[rr]^-{\Sigma \beta} 
&  & \Sigma B \ar[d]^{ \ \Sigma g}  \\ 
S^{11} \ar[rr]^-{\beta}&  &  B   \ar[rr]^-{i}
&  &    B \cup_{\beta}  e^{12}\ar[rr]^-{q} & & S^{12} \ar[rr]^-{\Sigma  \beta} 
&  & \Sigma B 
}$$ 
If $g_{\sharp}(\iota_{2}) = \lambda \iota_{2}$, \   $g_{\sharp}(\iota_{3}) = \mu \iota_{3}$  \  and \ $g_{\sharp}(h) \equiv \kappa h$ modulo 2-torsion for some integers $\lambda, \mu$  and $\kappa$ \ as in \cite[Section 4]{MimuraToda71}, then it is proved  \  $\kappa = \lambda^3$, \  $g_{\sharp}(\omega) = \lambda \mu \omega$ \  and 
\[g_{\sharp}(\beta) = a \lambda^{2}\mu^{4} [[[w, \iota_3],w], \iota_3] + b \lambda^{6}\mu [[\iota_3,h],h] + c \lambda^{5}\mu^{2} [[w,h],w] \]  
by the argument of \cite[Section 4]{MimuraToda71}.  
Moreover, by the above diagram,  we have 
$$g_{\sharp}(\beta) = s \beta = s(a[[[w, \iota_3],w], \iota_3] + b[[\iota_3,h],h] + c[[w,h],w]) .$$  

\noindent  
($\Longrightarrow$) \ 
We note that if $a=b=c=0$, then $N\calE (B\cup_{\beta} e^{12}) = N\calE (B \vee S^{12} )=12$. 
It follows that if $N\calE (B\cup_{\beta} e^{12}) = 3$, then at least one of $a,b,c$ is non-zero.  

\noindent  
($\Longleftarrow$) \ Conversely, assume that at least one of $a,b,c$ is non-zero. 

We see: if $a \neq 0$ then $\lambda^{2}\mu^{4} = s$; if $b \neq 0$ then $\lambda^{6}\mu = s$ and if $c \neq 0$ then $\lambda^{5}\mu^{2} = s$.  

If $\lambda = \pm 1$  and  $\mu = s = 0$, then $f_{\sharp} : \pi_{k} (B\cup_{\beta} e^{12}) \to \pi_{k} (B\cup_{\beta} e^{12})$ is an isomorphism for $k \leq 2$ and $f \not\in \calE (B\cup_{\beta} e^{12})$. Therefore,  $N\calE (B\cup_{\beta} e^{12}) \neq 2$. 

If $a \neq 0$ or  $b \neq 0$ or $c \neq 0$,  then  $\lambda^{2}\mu^{4} = s$ or  $\lambda^{6}\mu = s$ or $\lambda^{5}\mu^{2} = s$.  
Therfore, if $f_{\sharp} : \pi_{k} (B\cup_{\beta} e^{12}) \to \pi_{k} (B\cup_{\beta} e^{12})$ is an isomorphism for $k \leq 3$, then $\lambda = \pm 1$ and $\mu  = \pm 1$ and hence $s  = \pm 1$.  It follows that  
$g_{*} : H_{*} (X;\Z) \to H_{*} (X;\Z)$ is an isomorphism and hence 
$f_{*} : H_{*} (B\cup_{\beta} e^{12};\Z) \to H_{*} (B\cup_{\beta} e^{12};\Z)$ is an isomorphism by the homology exact sequence obtained by the diagram above and the five lemma. Hence we have $f \in \calE (B\cup_{\beta} e^{12})$  and  $N\calE (B\cup_{\beta} e^{12})=3$. 
\end{proof}

\begin{remark} \rm 
Assume that  all of $a,b,c$ are  non-zero in Theorem \ref{thm:MimuraToda}. Then, only the condition $\lambda \neq 0$ and $\mu \neq 0$ implies $f \in \calE(B\cup_{\beta} e^{12})$ as follows:  by an similar argument as in the proof of Lemma 4.1 of \cite{MimuraToda71}, we have 
$$s = \lambda^{2}\mu^{4} = \lambda^{6}\mu = \lambda^{5}\mu^{2} .$$ 
If $\lambda \neq 0$ and $\mu \neq 0$, then the relation $\lambda^{6}\mu = \lambda^{5}\mu^{2}$ implies $\lambda=\mu$. Therefore the relation $\lambda^{2}\mu^{4} = \lambda^{6}\mu$ implies $\lambda^{6} = \lambda^{7}$ and hence $\lambda =1$. It follows that $\lambda=\mu=1$ and $s = \lambda^{2}\mu^{4} = 1$ and we have $f \in \calE(B\cup_{\beta} e^{12})$. 
\end{remark}

\section{Rational homotopical properties}  \label{sec:RationalHomotopy} 

In this section, 
we assume that  spaces $X$ are {\it $1$-connected}  CW complexes of finite type.
The space  $X_0$ means  the rationalization of $X$ \cite{HMR}.
Then $\pi_*(X_0)=\pi_*(X)\otimes \Q$ and $H_*(X_0;\Z )=H_*(X;\Q )$. 
We assume familiarity with rational homotopy theory as in the text 
  \cite{FHT}.

Let $M(X)=(\Lambda {V},d)$ be the  Sullivan minimal model of $X$ \cite{S} and let $L(X)= (\LL U,\partial )$ be the Quillen minimal model of $X$  \cite{T}.
Since they are minimal, that is, the differentials are decomposable, 
$dV\subset \Lambda^{>1}V$ and    $\partial U\subset \LL^{>1}U$,  we have  
$V^n\cong {\rm Hom}(\pi_n(X),\Q ) \ \ and\ \ U^n\cong H_{n+1}(X;\Q )$, respectively \cite[p.34]{T}.
Then  
$[X_0,Y_0]\cong [M(Y),M(X)]$, which is the homotopy set in the category of differential graded algebras over $\Q$ and 
 $[X_0,Y_0]\cong [L(X),L(Y)]$, which is the homotopy set in the category of differential graded Lie algebras over $\Q$.

\begin{definition}\label{def:formal} \rm  (\cite{DGMS}) A space  $X$ is said to be {\it formal} if there is a quasi-isomorphism $M(X)\to (H^*(X;\Q ),0)$.
\end{definition}

\begin{example} \label{ex:formal} \rm 
It is known that the following spaces are formal: 
\begin{enumerate} 
\item[(a)]   the $n$-sphere $S^n$ \cite[\S 1]{HS}

\item[(b)]   the Eilenberg-Mac Lane space $K(G,n)$ $(n>1)$, where $G$ is a finitely generated abelian group.  

\item[(c)]   the one-point union of formal spaces \cite[Lemma 1.6]{HS}

\item[(d)]   the product space of formal spaces \cite[Example 2.87]{FOT}

\item[(e)]   symmetric  spaces \cite[p.162]{FHT}

\item[(f)]   compact K${\ddot{\rm a}}$hler manifolds \cite[Main Theorem]{DGMS}

\item[(g)]   $(m-1)$-connected compact $n$-manifolds  with $n\leq 4m-2$ for $m\geq 2$ \cite[Proposition 3.10]{FOT}

\item[(h)]   a space $X$ such that the differential of Quillen minimal model is quadratic \cite[II.7.(5)]{T} 

\item[(i)]   a space $X$ such that  $H^*(X;\Q )\cong \Q [x_1,..,x_n]/(f_1,..,f_n)$
where $f_1,..,f_n$ is a regular sequence \cite[II.7.(8)]{T}(\cite[Theorem 5]{Ha}).  
For example, homogeneous spaces $G/H$  with ${\rm rank}\, G={\rm rank}\, H$.\\
Remark that all the spaces  $B\cup e^{n}$ in Theorem 16 of \cite{OdaYama17} are formal. 
\end{enumerate}
\end{example}

\begin{remark}\label{formal-u} \rm  
Notice that for any commutative graded algebra $H$ over $\Q$,
there  \emph{uniquely} exists a minimal DGA $M$ and a quasi-isomorphism $M\to (H,0)$.
Then $M$ is the Sullivan model $M(X)$ of a  formal space $X$ with $H^*(X;\Q )=H$.  
Thus  a formal space having  a $\Q$-cohomology algebra $H$ is uniquely determined up to rational homotopy. 

\end{remark} 

\begin{remark} \rm 
If a map $f:M(X)\to M$ is a quasi-isomorphism for some DGA $M$, namely, $f^*:H^*(M(X))\cong H^*(M)$, then we have $[M(X),M(X)]=[M,M]$ (see \cite[Chapter X]{GM}).
Thus we have  $\mathcal{E}( M(X))=\mathcal{E}( M)$ (see \cite[p.305]{OdaYama18}).
For example, when $X$ is formal, we have $\mathcal{E}( M(X))=\mathcal{E}(H^*(X;\Q),0)$ and 
$N\mathcal{E}(M(X))=N\mathcal{E}(H^*(X;\Q),0)$.

\end{remark}

\begin{theorem}  \label{thm:HisoXYXformalNE}
Let $H^*(X;\Q )\cong H^*(Y;\Q)$ as graded algebras.
If $X$ is formal,  then  $N{\mathcal E}(X_0)\geq N{\mathcal E}(Y_0)$.
\end{theorem}
\begin{proof}
Let $(\Lambda Z,d)$ be the bigraded  model of   $H^*(X;\Q )$
\cite[p.242]{HS}, in which $d$ is minimal.
Then the (non-minimal)  model of $Y$ is given as the filtered model  $(\Lambda Z,D=d+\sigma )$ for some derivation $\sigma:\Lambda Z\to \Lambda Z$
of degree $+1$ \cite[Theorem 4.4]{HS}.
In particular, $\sigma$ is zero when $Y$ is formal.

Let $n=N{\mathcal E}(X_0)$.
Suppose  that a map $f:Y_0\to Y_0$ induces $\pi_{\leq n}(f):\pi_{\leq n}(Y_0)\cong \pi_{\leq n}(Y_0)$.   
Then we have
\begin{equation} 
f^*:H^{\leq n}(Y;\Q)=H^{\leq n}(\Lambda Z, D )\cong H^{\leq n}(Y;\Q)=H^{\leq n}(\Lambda Z, D ) \tag{$*$}
\end{equation}  
by  the rational Whitehead theorem  ($f^*:H^{n}(Y;\Q) \to  H^{n}(Y;\Q)$ is surjective, and hence bijective).

Let  $g:(\Lambda Z,d )\to (\Lambda Z,d )$ be the map of minimal models
induced by  the DGA-map $$(f^*,0):(H^{*}(Y;\Q),0)=(H^{*}(\Lambda Z, D ),0)\to
(H^{*}(Y;\Q),0)= (H^{*}(\Lambda Z, D ),0).$$
Then $g^*=f^*$ and therefore we have  $g|_{Z^{\leq n}}:Z^{\leq n}\cong Z^{\leq n}$ from $(*)$.
 It is  obtained by  the rational Whitehead's theorem  since $Z^i$ is the dual space of $\pi_i(X_0)$.
 From the assumption,  $g$ is a DGA-isomorphism.
Thus $f^*(=g^*)$ is an isomorphism.
Then $f$ is a homotopy equivalence or $f\in{\mathcal E}(Y_0)$.
\end{proof}

\begin{proposition}\label{ee}  
Let $n \geq 3$ and $B$  a $1$-connected CW-complex with    $H_{*}\mbox{-}\dim (B) \leq n-2$. Let $\gamma  :S^{n-1} \to B$ be a generator   of a direct summand $\Z \subset \pi_{n-1} (B)$. 
Then the following are equivalent\/$:$  
\begin{enumerate} 
\item[{\rm (i)}] \    $N\mathcal{E}( (B\cup_{\gamma} e^n)_0)=n$.
\item[{\rm (ii)}] \   $(B\cup_{\gamma} e^n)_0 \sim (B\vee S^n)_0$. 
\end{enumerate}   
\end{proposition}
\begin{proof}
(i) $\Longrightarrow$ (ii) holds by   Theorem \ref{thm:17Thm3Imporved} (cf.  \cite[Theorem 15]{OdaYama17}). \\ (ii) $\Longrightarrow$  (i) is obvious. 
\end{proof}

\begin{remark}\label{subset}  \rm  
In Proposition \ref{ee}, the condition  $\mathcal{E}( B_0)\subsetneq \mathcal{E}( X_0)$ is not equivalent to (ii). For example, 
let $B=S^2\vee S^3$ and $L (B)=(\LL (u_1,u_2),0)$
with $|u_i|=i$.
Let $X=B\cup e^4=S^3\vee \C P^2$ 
 and 
$L (X)=(\LL (u_1,u_2,u_3),\partial )$
with $|u_3|=3$, $\partial (u_1)=\partial (u_2)=0$ and $\partial (u_3)=[u_1,u_1]$.
Then there is a map $f\in \calE (L(X) )$ such that
$f(u_3)=au_3+b[u_1,u_2]$ for any $a\in \Q^*$ and $b\in \Q$.
Thus we have  
$$\Q^*\times \Q^*=\{ (c_1,c_2)\mid c_i\in \Q^*\}\cong \mathcal{E}( B_0)\subsetneq \mathcal{E}( X_0)\cong \{ (c_1,c_2,b)\}= \Q^*\times \Q^*\times \Q$$
where $a=c_1^2$ for $f(u_i)=c_iu_i$ ($i=1,2$) 
but $X_0\not\sim (B\vee S^4)_0$.
\end{remark}

\begin{theorem}\label{formalen}
Let $n \geq 3$ and $B$  a $1$-connected CW-complex with    $H_{*}\mbox{-}\dim (B) \leq n-2$. Let $\gamma  :S^{n-1} \to B$ be a generator   of a direct summand $\Z \subset \pi_{n-1} (B)$.  Let $B$ be a formal space with $2H_{*}\mbox{-}\dim (B)  <n$.
Then 
 $X=B\cup_{\gamma} e^n$ is formal  if and only if $N\mathcal{E}( X_0)=n$.  
\end{theorem}
\begin{proof}
 The `if part' is obvious by Proposition \ref{ee} and Example \ref{ex:formal}(c).
The `only if part' holds as follows:
Suppose that $N\mathcal{E}( X_0) \neq n$. Then  the condition (ii) of Proposition \ref{ee}   does not hold. 
Then, since $H^*(B\cup_{\gamma} e^n;\Q )\cong H^*(B\vee S^n;\Q )$ as graded algebras from $(H^*(B;\Q ) \otimes H^*(B;\Q ))^{n}=0$ and  $(B\cup e^n)_0\not\sim (B\vee S^n)_0$, in which  $B\vee S^n$ is formal, we see that $X$ is not formal  by Remark \ref{formal-u}. 
\end{proof}

\begin{remark}\label{formal}  \rm  
The condition of Theorem \ref{formalen} is best possible.
For example: Let $B=S^{2}$, $X=\C P^2$ and $n=4$. Then,  $2 H_{*}\mbox{-}\dim (B)  =4=n$ and  $X=\C P^2$ is formal but $N{\mathcal E}(X_0) =N{\mathcal E}((S^2\cup_{[\iota_2,\iota_2]} e^4)_0) =2$.  We note that  $\C P^2_0\not\sim (S^2\vee S^4)_0$. 
\end{remark}

\begin{example}  \label{exa:threecell}  \rm 
Let $B=S^3 \vee S^5$ with $x$ and $y$ the homotopy generators of $\pi_3(S^3)\otimes \Q$ and 
$\pi_5(S^5)\otimes \Q$, respectively.
Let $B\cup_{\alpha}e^{12}$  be a CW complex with 
\begin{center} 
$\alpha =[x,[x,[x,y]]] + [y,[x,y]] : S^{11} \to B$ \ \  (Whitehead product). 
\end{center} 
Then 
\begin{equation}  
 H^*(B\cup_{\alpha}e^{12};\Q )\cong H^*(S^3 \vee S^5\vee S^{12};\Q )  \tag{$*$} 
\end{equation}
as graded algebras. 
However  
$$N\calE ((B\cup_{\alpha}e^{12})_0)=  N\calE (B_0)=5<N\calE ((S^3 \vee S^5\vee S^{12})_0)=12$$
since  $\calE ((B\cup_{\alpha}e^{12})_0)=\calE (L(B)* \LL (w),\partial )=\calE (\LL (u,v,w),\partial )\ (\cong \Q^* )$
with $|u|=2$, $|v|=4$, $|w|=11$, $\partial(u)=\partial (v)=0$ and (\cite[p.84]{T})  
$$\partial (w)=[u,[u,[u,v]]] + [v,[u,v]].$$   
Here $\LL (V)*\LL (U)=\LL (V\oplus U)$ (free product \cite[p.17]{T}).
Indeed,  for any map 
$$f \in [(B\cup_{\alpha}e^{12})_0 , (B\cup_{\alpha}e^{12})_0] = [\LL (u,v,w) , \LL (u,v,w)],$$ 
we may set 
$$f(u)=\lambda u , \ \  f(v)=\mu v  \ \ \mbox{and} \ \    f(\omega)=\kappa  \omega \ \ \  \ (\lambda, \mu , \kappa  \in \Q)$$ 
by the dimensional reason.  If $f_{\sharp} : \pi_{3}((B\cup_{\alpha}e^{12})_0)  \to \pi_{3}((B\cup_{\alpha}e^{12})_0)$ is an isomorphism, then we have $\lambda \neq 0$. We see 
\begin{eqnarray*} f(\partial \omega) &=& f([u,[u,[u,v]]] + [v,[u,v]]) \\ 
 &=& [f(u),[f(u),[f(u),f(v)]]] + [f(v),[f(u),f(v)]]     \\ 
  &=& \lambda^3\mu [u,[u,[u,v]]] + \lambda \mu^2 [v,[u,v]]    \ \ \ \mbox{and} \\ 
\partial (f(\omega)) &=& \partial (\kappa  z) = \kappa \partial (\omega)   = \kappa ([u,[u,[u,v]]] + [v,[u,v]]). 
\end{eqnarray*} 
Since  $f(\partial \omega) = \partial (f(\omega))$, we have  $\kappa = \lambda^3\mu=\lambda \mu^2$.  Therefore, the case $\lambda \neq 0$ and $\mu= \kappa = 0$ can happen, and we conclude that $N\calE ((B\cup_{\alpha}e^{12})_0)=5$.

Finally, we see $ {\mathcal E}(B_0)\cong \Q^*\times \Q^*$  and if $f \in {\mathcal E}((B\cup e^n)_0)$ then  $\lambda \neq 0$ and $\mu \neq 0$, and hence $\lambda^2=\mu$. It follows that 
$${\mathcal E}((B\cup e^n)_0)= \{ (\lambda , \mu , \kappa  ) \in\Q^* \times \Q^*  \times \Q^* \mid \mu=\lambda^2, \ \ \kappa = \lambda^5  \}\cong \Q^*.$$
Notice that  $S^3 \vee S^5\vee S^{12}$ is formal but  $B\cup_{\alpha}e^{12}$ is not by $(*)$ and Remark \ref{formal-u}.

Therefore, if $B=S^3\vee S^5$ and $\alpha =[x,[x,[x,y]]] + [y,[x,y]]$, then the following are equivalent:
\begin{itemize}
\item[{\rm (a)}]   $B\cup_{\alpha}e^{12}$ is formal, 
\item[{\rm (b)}]   $(B\cup_{\alpha}e^{12})_0\simeq (S^3 \vee S^5\vee S^{12})_0$, 
\item[{\rm (c)}]    ${\mathcal E}(B_0)\subsetneq {\mathcal E}((B\cup_{\alpha} e^n)_0)$, 
\item[{\rm (d)}]    ${\mathcal E}((B\cup_{\alpha} e^n)_0)\cong \Q^*\times \Q^*\times \Q^*$ and 
\item[{\rm (e)}]   $N\calE ((B\cup_{\alpha}e^{12})_0)=12$.
\end{itemize}

In general, if $\beta =a[x,[x,[x,y]]]+b[y,[x,y]]$ for $a,b\in \Q$, then we have the following table:  
\begin{center} \small 
\setlength{\extrarowheight}{1ex} 
\hspace*{-2ex}\begin{tabular}{|l||c|c|c|}  
\hline  \hspace*{12ex} {case} & 
{self-equivalences} & 
{formal} &
{$N\calE ((B\cup_{\beta} e^{12})_0)$}   \\[0.4ex]    
\hline 
(i) \  \ $ab\neq 0$ & $\calE ((B\cup_{\beta} e^{12})_0)\subsetneq \calE (B_0)$ & no & $5$ \\[0.4ex] 
\hline
(ii) \ $a\neq 0$, $b=0$ or $a=0$, $b\neq 0$ &  $\calE ((B\cup_{\beta} e^{12})_0)= \calE (B_0)$ & no & $5$ \\[0.4ex] 
\hline 
(iii) \ $a=b=0$ &  $\calE ((B\cup_{\beta} e^{12})_0)\supsetneq \calE (B_0)$ & yes & $12$ \\[0.4ex]  
\hline
\end{tabular} 
\end{center}  \normalsize 

\end{example}

\vspace{1ex}

Next we consider  closed (compact without boundary) manifolds.

\begin{lemma}\label{mfd} Let $X$ be a $1$-connected  closed  $n$-manifold.
Then a map $f:X\to X$ preserve the fundamental class  $\Omega \in H^n(X;\Q )$  up to scalar multiple  $($that is, $f^*(\Omega )=\lambda \Omega$ for   $\lambda \neq 0\in \Q \,)$  if and only if  the rationalized map $f_0$ is a homotopy equivalence, namely  $f_0\in {\mathcal E}(X_0)$.
\end{lemma}
\begin{proof} 
Since `if part' is obvious, we have to prove `only if part':  For any $a\neq 0\in H^k(X;\Q )$, there is an element $b\in H^{n-k}(X;\Q )$ such that $a\cdot b=\Omega$
by the rational Poincar\'{e} duality \cite[Definition 3.1]{FOT}.
If $f^*(a)=0$, we have $f^*(\Omega )=f^*(a\cdot b)=f^*(a)f^*(b)=0$.
It contradicts the assumption.
Thus $f^*(a)\neq 0$, and hence, $f^*:H^*(X;\Q )\to H^*(X;\Q)$ is isomorphic.
Therefore $f_0:X_0\to X_0$ is a homotopy equivalence. 
\end{proof}

\begin{corollary}\label{vee}
Let $X=\vee_{i=1}^kX_i$ be the one-point union of  $1$-connected   closed  $n$-manifolds $X_i$
with  the fundamental classes  $\Omega_i$.
Then a map $f:X\to X$ 
induces 
$f^*(\Omega_i)=\sum_{j=1}^{k} a_{ij}\Omega_j$  $(a_{ij}\in\Q )$ for $i=1,..,k$ with $\det (a_{ij})\neq 0$ for the $k\times k$-matrix $(a_{ij})$
if and only if    $f_0\in {\mathcal E}(X_0)$.
\end{corollary}

\begin{remark} \rm 
A space $X$  is called a Poincar\'e duality space if $H^{*}(X; \Q)$ satisfies the Poincar\'e duality \cite[p.511]{FHT}. 
The results of Lemma \ref{mfd} and Theorem \ref{thm:Xsymplecticmfd} are generalized to   Poincar\'e duality spaces. See Proof of Lemma \ref{lem:orderOfgamma}. 
\end{remark}

Let  $X$ and $Y$ be  $1$-connected  closed  $n$-manifolds.
Then the connected sum $X\sharp Y$ of $X$ and $Y$ \cite[p. 108]{FOT}  is also an $n$-manifold.

\begin{theorem}  \label{thm:NEXsharpYleqvee}  $N{\mathcal E}((X\sharp Y)_0)\leq N{\mathcal E}((X\vee Y)_0)$ for $1$-connected  closed  $n$-manifolds $X$ and $Y$. 
\end{theorem}
\begin{proof}
 Let $M(X)=(\Lambda V_X,d_X)$ and $M(Y)=(\Lambda V_Y,d_Y)$. 
There is a DGA-map induced by the pinching map $p: X\sharp Y\to X\vee Y$
$$\psi :(M(X\vee Y)\otimes \Lambda (v),d)=(\Lambda (V_X\oplus V_Y\oplus Z\oplus \Q (v)),d)\to M(X\sharp Y)$$
where $d|_{V_X}=d_X$,  $d|_{V_Y}=d_Y$, $|v|=n-1$ and $d(v)=w_X-w_Y$.
Here $w_X$ and $w_Y$ are cocycles representing the fundamental classes of $M(X)$ and $M(Y)$, respectively and 
$Z$ is made of generators inductively requiring the multiplications $x\cdot y=0$ for $x\in H^*(X;\Q )$ and $y\in H^*(Y;\Q )$.
Due to  \cite[Example 3.6]{FOT},  $H^{\leq n}(\psi )$ is an isomorphism, that is, 
\begin{flushright} 
$\psi$ is the $(n)$-minimal model of $X\sharp Y$.  \hspace*{30ex} $(*)$
\end{flushright} 
Suppose that $m:=\min \{ N{\mathcal E}((X\vee Y)_0),n-2\}$. 
Here  $N{\mathcal E}((X\sharp Y)_0)\leq n-2$
by Theorem \ref{thm:XneqiSn}.

 For $f\in [(X\sharp Y)_0,(X\sharp Y)_0]=[ M(X\sharp Y),M(X\sharp Y)   ]$, 
assume that  $$f_{\sharp}: \pi_{\leq m}((X\sharp Y)_0)\cong \pi_{\leq m}((X\sharp Y)_0)
.$$
Then $f^*:(V_X\oplus V_Y\oplus Z\oplus \Q (v))^{\leq m}\cong (V_X\oplus V_Y\oplus Z\oplus \Q (v))^{\leq m}$.
Notice that $f^*(v)=av+g$ for some $a\in \Q-0$ and $g\in \Lambda (V_X\oplus V_Y\oplus Z)$  if $n-2 = m$.
Thus we obtain 
$f^*:(V_X\oplus V_Y\oplus Z)^{\leq m}\cong (V_X\oplus V_Y\oplus Z)^{\leq m}$. 

Let $\overline{f} : (X\vee Y)_0 \to (X\vee Y)_0$ be the map induced by $f :(X\sharp Y)_0  \to (X\sharp Y)_0$ and the pinching map $p: X\sharp Y\to X\vee Y$ as in the following commutative diagram:    
$$\xymatrix{
 (X\sharp Y)_0 \ar[d]_{p_0 \ } \ar[rr]^-{f} & &    (X\sharp Y)_0 \ar[d]^{\  p_0}       \\ 
(X\vee Y)_0 \ar[rr]^-{\overline{f} } & &    (X\vee Y)_0  
}$$
From the assumption, we have $\overline{f}^*:M( X\vee Y)\cong M( X\vee Y)$.
Then $\overline{f}^*(w_X-w_Y)=bw_X-cw_Y$ for some $b,c\in \Q$ where $b\neq 0$ or $c\neq 0$ from Corollary \ref{vee}.  We see: 
\begin{center} 
$\overline{f}^* (d(v)) = \overline{f}^* (w_X-w_Y)=bw_X-cw_Y + \alpha$ \ for an element $\alpha$ with $\dim (\alpha) = n$;  \vspace{1ex}  \\  
$d(\overline{f}^* (v)) = d(av+g) = ad(v) + d(g) = a(w_X-w_Y) + d(g)$.  
\end{center} 
Therefore, by the relation $\overline{f}^* (d(v))=d(\overline{f}^*(v))$,  we have $a=b=c\ (\neq 0)$, since non-exact cocycles of $M(X \vee Y)^n$ are only $w_X$ and $w_Y$ (Here, an element $\alpha$ is said to be {\it exact} if there exists an element $\beta$ such that $d(\beta) = \alpha$).
Therefore we have 
$$f^*:(\Lambda (V_X\oplus V_Y\oplus Z\oplus \Q (v)),d)\cong (\Lambda (V_X\oplus V_Y\oplus Z\oplus \Q (v)),d),$$ 
which induces $H^{\leq n}(f): H^{\leq n}((X\sharp Y)_0)\cong H^{\leq n}((X\sharp Y)_0)$
from the property $(*)$ of $\psi$.
Since $H^{>n}( X\sharp Y;\Q )=0$, we obtain $f \in {\mathcal E}((X\sharp Y)_0)$. 
\end{proof}

\begin{example}  \rm 
\begin{enumerate} 
\item[{\rm (1)}] \ We consider the case where   $H^*(X;\Q )$ and $H^*(Y;\Q )$ are generated by only one element as follows:   Let $M(X)=(\Lambda (x,x'),d_X)$ 
with $|x|$ even, $d_X(x)=0$, $d_X(x')=x^k$ for some $k \geq 2$ and $M(Y)=(\Lambda (y,y'),d_Y)$ with $|y|$ even, $d_Y(y)=0$, $d_Y(y')=y^m$ for some $m \geq 2$.
Here $(k-1)|x|=(m-1)|y|=n$.
Then $M(X\sharp Y)\cong (\Lambda (x,y,u,v),d)$
with $d(x)=d(y)=0$, $d(u)=xy$ and $d(v)=x^{k-1}-y^{m-1}$ by an  argument similar to the proof of Theorem \ref{thm:NEXsharpYleqvee}. 
When $|x|\leq |y|$, we have $N{\mathcal E}((X\sharp Y)_0)=|x|\leq |y|=  N{\mathcal E}((X\vee Y)_0)$
since $f^*(y)\neq 0$ if $f^*(x)\neq 0$ for  any map $f: X \to X$. 
For example, $X = \C P^{2n}$ and $Y = \HH P^n$. 
\item[{\rm (2)}] \  If $X$ is a $1$-connected $n$-manifold and $Y_0\simeq S^{n}_0$ for some  $n  \geq 2$, then we have  $(X\sharp Y)_0\simeq X_0$.
Thus  $N{\mathcal E}((X\sharp Y)_0)= N{\mathcal E}(X_0) \leq n=  N{\mathcal E}((X\vee Y)_0)$. 
For example, $X = \C P^{n}$ and $Y = S^{2n}$.
\end{enumerate} 
\end{example}

A $2n$-manifold $X$ is said to be {\it symplectic} if it possesses a nondegenerate $2$-form $w$ which is closed, that is, $dw=0$ \cite[Definition 4.76]{FOT}. 
The nondegeneracy condition is equivalent to saying that $w^n$ is a true volume  form on $X$ \cite[p.182]{FOT}.
For example,  K\"{a}hler manifolds are symplectic \cite[Example 4.80]{FOT}. 
The reader is referred to  \cite{Ba} and \cite{FOT} for the examples of (non-formal) $1$-connected   symplectic manifolds.
A $2n$-manifold $X$ is {\it cohomologically-symplectic} (or {\it c-symplectic}) if there is a class  $\omega \in H^{2}(X; \Q)$ such that $\omega^{n}$ is a top class for $X$ (Definition on p.263 of Lupton and Oprea \cite{LupOprea95}, Definition 4.76 (p.182) of F\'{e}lix,  Oprea  and    Tanr\'{e} \cite{FOT}).  
A symplectic manifold is a cohomologically-symplectic manifold. 

\begin{theorem} \label{thm:Xsymplecticmfd}
If $X$ is a $1$-connected   c-symplectic manifold,  then  $N{\mathcal E}(X_0)= 2$.
\end{theorem}
\begin{proof}
When $\dim X=2n$, the fundamental class is given as 
$\Omega =[w^n]$ for an element $w  \in H^{2}(X; \Q)$ since $X$ is c-symplectic.  Then the result follows by Lemma \ref{mfd}. 
\end{proof}

For any elements $x,y,z$ of  an graded Lie algebra $L$, the Lie bracket $[\ ,\ ]:L\otimes L\to L$ satisfies
\begin{enumerate} 
\item[{\rm (i)}] \    antisymmetry: $[x,y]=(-1)^{|x||y|+1}[y,x]$
\item[{\rm (ii)}] \  Jacobi identity: $(-1)^{|x||z|}[x,[y,z]]+ (-1)^{|y||x|}[y,[z,x]]+ (-1)^{|z||y|}[z,[x,y]]=0$ (\cite[p.14]{T}).
\end{enumerate}

\begin{example}\label{Exa:fourCellcpx} \rm (Four cell complex)
Let $B=S^3\vee S^2$ and $\iota_2 : S^2 \to  S^3 \vee  S^2 =B$ and $\iota_3 : S^3 \to  S^3 \vee  S^2 =B$  be the inclusions. Let  $X=B\cup_{\alpha }e^{4}\cup_{\beta}e^6$
for the attaching maps: 
$$\alpha =a [\iota_2,\iota_2] : S^3 \to S^2 \subset  S^3 \vee  S^2 =B,$$
$$\beta =b [[\iota_2,\iota_3 ],\iota_2]+ch_5 : S^5 \to  S^3 \vee  \C P^2 =B\cup_{\alpha }e^{4} \ \ (a=1)$$
with $a,b,c \in \Z$.
Here $h_5:S^5\to \C P^2$ is the Hopf map and  we define  $c=0$ if $a\neq 1$. 
Let $\mathcal{E}_{*}(X)$ be the subgroup of $ \mathcal{E}(X)$ whose elements  induce the identities on homology groups  (\cite{AM}), that is,  
\[\mathcal{E}_{*}(X):=\{ f\in \mathcal{E}(X)\mid \ H_*(f)={\rm id}_{H_*(X)} \}.\]

Then we obtain the following table, where the numbers (1) -- (4) in the table are dicussed after the table:
\begin{center} \small
\setlength{\extrarowheight}{1ex} 
\hspace*{-5ex}\begin{tabular}{|l||c|c|c|c|}  
\hline  \hspace*{10ex}{case} & 
{$\mathcal{E}( X_0)/\mathcal{E}_{*}( X_0)$ } &$X$ & formal&
{$N\calE (X)$}   \\[0.4ex]    
\hline
(i)  \ \ \  $a=b=c=1$ & \hspace*{2ex} $ \Q^*  \ \  {\cdots \mbox{\footnotesize (2)}}$ &$S^3\vee \C P^2\cup_{ [[\iota_2,\iota_3 ],\iota_2]+h_5}e^6$  & no \ ${\cdots \mbox{\footnotesize (4)}}$ &   2${\cdots \mbox{\footnotesize (1)}}$ \\[0.4ex] 
\hline
(ii) \ \  $a=c=1$, $b=0$ &  $(\Q^*)^{\times 2} {\cdots \mbox{\footnotesize (2)}}$ & $S^3\vee \C P^3$ &  yes ${\cdots \mbox{\footnotesize (3)}}$ & $3$ \\[0.4ex] 
\hline 
(iii) \ $a=c=0$, $b=1$ &  $(\Q^*)^{\times 3} {\cdots \mbox{\footnotesize (2)}}$ &$(S^3\vee S^2)\cup_{ [[\iota_2,\iota_3 ],\iota_2]} e^6\vee  S^4$ & no \ ${\cdots \mbox{\footnotesize (4)}}$ &  $4$ ${\cdots \mbox{\footnotesize (1)}}$ \\[0.4ex]  
\hline
(iv) \ $a=b=c=0$ &  $(\Q^*)^{\times 4} {\cdots \mbox{\footnotesize (2)}}$ &$S^3\vee S^2\vee S^4\vee S^6$ &  yes ${\cdots \mbox{\footnotesize (3)}}$ &  $6$ \\[0.4ex]  
\hline
\end{tabular} 
\end{center}   \normalsize

\noindent (1) 
When $f_{\sharp}(\iota_2)=\lambda \iota_2$, $f_{\sharp}(\iota_3)=\mu \iota_3$ ($\lambda,\mu\in \Z$)
for a map $f:X\to X$, we have 
$$f_{\sharp}(  [[\iota_2,\iota_3 ],\iota_2])= \lambda^2\mu [[\iota_2,\iota_3 ],\iota_2] \ \ \ \  \mbox{and}   \ \ \ \ f_{\sharp}(h_5)=\lambda^3h_5 .$$ 
Here, the last relation is obtained by Proposition 2 of \cite{OdaYama17}. \\ 
The case (i):  \ Suppose $abc\neq 0$.
Then if $\lambda \neq 0$, we have $\mu\neq 0$ from $\lambda^2\mu=\lambda^3$ by Proposition 2 of \cite{OdaYama17} making use of a similar argument as in the proof of Theorem  \ref{thm:MimuraToda}, that is, the relation 
\[ s([[\iota_2,\iota_3 ],\iota_2]+ h_5) =  \lambda^2\mu [[\iota_2,\iota_3 ],\iota_2]    + \lambda^3h_5  \ \ \  (\mbox{for some integer} \  s) \] 
implies  $s = \lambda^2\mu = \lambda^3$. 
Thus we have $N\calE (X)=2$.  \\ 
The case (iii): If $a=c=0$ and $b=1$, then we have $s= \lambda^2\mu$. Therefore, if $\lambda \neq 0$ and $\mu\neq 0$, then $s \neq 0$. Then we have $N\calE (X)=3$.

\noindent (2) We determine  $\mathcal{E}( X_0)/\mathcal{E}_{*}( X_0)$ by making use of the Quillen models:  $L(X)$ is given by $(\LL (u_1,u_2,u_3,u_5),\partial ) $ with $|u_i|=i$,  $\partial (u_1)=\partial (u_2)=0$, and \\ 
(i) \ \ \  $\partial (u_3)=[u_1,u_1]$ and  $\partial (u_5)=[[u_1,u_2],u_1]+[u_1,u_3]$, \\
(ii) \ \  $\partial (u_3)=[u_1,u_1]$ and  $\partial (u_5)=[u_1,u_3]$,  \\
(iii) \ $\partial (u_3)=0$ and  $\partial (u_5)=[[u_1,u_2],u_1]$, \\
(iv) \  $\partial (u_3)=0$ and  $\partial (u_5)=0$.\\
Let a DGL-map $f:L(X)\to L(X)$ be given by 
$\ f(u_i)=\lambda_i u_i$ 
 for $\lambda_i\in\Q^*$ $(i=1,2,3,5)$. 
Then, $\mathcal{E}( X_0)/\mathcal{E}_{*}( X_0)$ is determined as follows:

\noindent 
The case  (i): We have  $\lambda_3=\lambda_1^2$ and  $\lambda_5=\lambda_1^2\lambda_2=\lambda_1\lambda_3$  from $\partial\circ f=f\circ \partial$.
If $f\in \calE(L (X))$, then we may set $\lambda:=\lambda_1=\lambda_2\neq 0$.
Thus we have 
$\calE(L (X))/\mathcal{E}_{*}( L(X))\cong \Q^*$ with elements $f$ such that
$$\ f(u_1)=\lambda u_1,  f(u_2)=\lambda u_2,  f(u_3)=\lambda^2 u_3,  f(u_5)=\lambda^3 u_5\ \ \ {\rm for}  \ \ \ \lambda\in\Q^*.
$$

\noindent 
The case  (ii):   $\lambda_3=\lambda_1^2$ and  $\lambda_5=\lambda_1\lambda_3$  from $\partial\circ f=f\circ \partial$.
Thus  we have 
$\calE(L (X))/\mathcal{E}_{*}( L(X))=\{ (\lambda_1,\lambda_2)\}\cong (\Q^*)^{\times 2}$ with elements $f$ such that
$$\ f(u_1)=\lambda_1 u_1,  f(u_2)=\lambda_2 u_2,  f(u_3)=\lambda_1^2 u_3,  f(u_5)=\lambda_1^3 u_5\ \ \ {\rm for}  \ \ \ \lambda_i\in\Q^*.
$$

\noindent 
The case  (iii):   $\lambda_5=\lambda_1^2\lambda_2$  from $\partial\circ f=f\circ \partial$.
Thus
 we have 
$\calE(L (X))/\mathcal{E}_{*}( L(X))=\{ (\lambda_1,\lambda_2,\lambda_3)\}\cong (\Q^*)^{\times 3}$ with elements $f$ such that
$$\ f(u_1)=\lambda_1 u_1,  f(u_2)=\lambda_2 u_2,  f(u_3)=\lambda_3 u_3,  f(u_5)=\lambda_1^2 \lambda_2 u_5\ \ \ {\rm for} \ \ \ \lambda_i\in\Q^*.
$$

\noindent 
The case (iv): there is no restriction for $\{ \lambda_i\}$ from $\partial\circ f=f\circ \partial$.
Thus
 we have 
$\calE(L (X))/\mathcal{E}_{*}( L(X))=\{ (\lambda_1,\lambda_2,\lambda_3,\lambda _5)\}\cong (\Q^*)^{\times 4}$ with elements $f$ such that
$$\ f(u_1)=\lambda_1 u_1,  f(u_2)=\lambda_2 u_2,  f(u_3)=\lambda_3 u_3,  f(u_5)=\lambda_5 u_5\ \ \ {\rm for}  \ \ \ \lambda_i\in\Q^*.
$$
\noindent
Remark that the map $g\in \calE(L (X))$ with $g(u_3)=u_3+[u_1,u_2]$ and $g(u_i)=u_i$ for the other $i$ is  an element of $\mathcal{E}_{*}( L(X))$.

\noindent 
(3) By Example \ref{ex:formal}(c),    the one-point union of formal spaces is formal.

\noindent 
(4) $H^*(X;\Q )=\Lambda (x)\otimes \Q [y,z]/(xy,xz,y^3,yz,z^2)=H^*(S^3\vee \C P^2\vee S^6;\Q )$
where $|x|=3$, $|y|=2$ and $|z|=6$. Since $S^3\vee \C P^2\vee S^6$ is formal and $X_0\not\sim (S^3\vee \C P^2\vee S^6)_0$,  the space $X$ is not formal by Remark \ref{formal-u}.   
\end{example}

\begin{example} \label{Exa:fourAttachingMaps} 
\rm (Four attaching maps of a cell)  \ 
Let $B=S^3\vee \C P^4$  and $\iota_2 : S^2 \to  \C P^4 \to S^3\vee \C P^4=B$ and $\iota_3 : S^3 \to   S^3\vee \C P^4=B$  be the inclusions.  Let $h_9 : S^9 \to S^9 / S^1 = \C P^4 \subset S^3\vee \C P^4=B$ be the natural map.  Let $X=B\cup_{\alpha}e^{10}$  for the attaching map 
$$\alpha =a  [[[\iota_2, \iota_3],\iota_3],[\iota_2,\iota_3]] +b h_9 : S^9 \to  B=S^3\vee \C P^4$$
with $a,b\in \Z$.
Let  $f:X\to X$  be a map and $g = f|_B : B \to B$ be the restriction  of $f$.  When $g_{\sharp}(\iota_2)=\lambda \iota_2$, $g_{\sharp}(\iota_3)=\mu \iota_3$ ($\lambda,\mu\in \Z$), we have 
$$g_{\sharp}( [[[\iota_2, \iota_3],\iota_3],[\iota_2,\iota_3]])= \lambda^2\mu^3[[[\iota_2, \iota_3],\iota_3],[\iota_2,\iota_3]] \ \  \ \mbox{and}  \ \ \ g_{\sharp}(h_9)=\lambda^5h_9$$
The last relation is obtained by Proposition 2 of \cite{OdaYama17}. Now, we determine $N\calE (X)$.   

\noindent 
The case (i): Suppose $ab\neq 0$.
Then if $\lambda \neq 0$, we have $\mu\neq 0$ from $\lambda^2\mu^3=\lambda^5$ by Proposition 2 of \cite{OdaYama17} as in the proof of Theorem  \ref{thm:MimuraToda}, and hence we have $N\calE (X)=2$.  \\  
The case (ii): Let $H_{2} (X) = \Z \{x_{2}\}$, $H_{3} (X) = \Z \{x_{3}\}$, $H_{10} (X) = \Z \{x_{10}\}$. Suppose that $f_{*} (x_{2}) = \lambda x_{2}$, $f_{*} (x_{3}) = \mu x_{3}$ and $f_{*} (x_{10}) = \kappa x_{10}$. Then we have a relation $k = \lambda^{2} \mu^{3}$  by Proposition 2 of \cite{OdaYama17}. It follows that $N\calE (X) = 3$.

We obtain the following table:

\begin{center} \small
\setlength{\extrarowheight}{1ex} 
\begin{tabular}{|l||c|c|c|c|c|c|}  
\hline   \hspace*{5ex} {case} & 
{$\mathcal{E}( X_0)/\mathcal{E}_{*}( X_0)$ }&$X$  &$H^*(X;\Q )$ &
{formal}&  {$N\calE (X)$}   \\[0.4ex]    
\hline 
(i) \ \ \ $ab\neq 0$ & \hspace*{1ex} $\Q^*$ \ \ \ ${\cdots \mbox{\footnotesize (5)}}$& $(S^3 \vee \C P^4)\cup_{\alpha}e^{10}$  & (1) & no \ ${\cdots \mbox{\footnotesize (4)}}$  &  $2$ \\[0.4ex] 
\hline
(ii) \ \ $a\neq 0$, $b=0$ &  $(\Q^*)^{\times 2}$ ${\cdots \mbox{\footnotesize (5)}}$& $(S^3 \vee \C P^4)\cup_{\alpha}e^{10}$ &(1)  & no \ ${\cdots \mbox{\footnotesize (4)}}$  & $3$ \\[0.4ex] 
\hline
(iii) \ $a=0$, $b\neq 0$ &  $(\Q^*)^{\times 2}$ ${\cdots \mbox{\footnotesize (5)}}$& $S^3 \vee \C P^5$ & (2)& yes ${\cdots \mbox{\footnotesize (3)}}$  &  $3$ \\[0.4ex] 
\hline 
(iv) \ $a=b=0$ &  $(\Q^*)^{\times 3}$ ${\cdots \mbox{\footnotesize (5)}}$& $S^3 \vee \C P^4  \vee S^{10}$ &(1)&  yes ${\cdots \mbox{\footnotesize (3)}}$  &  $10$ \\[0.4ex]  
\hline
\end{tabular} 
\end{center}   \normalsize 
where
(1) $=\Lambda (x)\otimes \Q [y,z]/(xy,xz,y^5,yz,z^2)$ and
(2) $=\Lambda (x)\otimes \Q [y]/(xy,y^6)$
with $|x|=3$, $|y|=2$ and $|z|=10$.

\noindent 
(3) By Example \ref{ex:formal}(c),    the one-point union of formal spaces is formal.

\noindent 
(4)  See Remark \ref{formal-u} for formal  spaces.

\noindent (5) 
$L(X)$ is given by $(\LL (u_1,u_2,u_3,u_5,u_7,u_9),\partial )$ with $|u_i|=i$,   $\partial (u_1)=\partial (u_2)=0$, and \\ 
(i) \ \   $\partial (u_3)=[u_1,u_1]$,  $\partial (u_5)= 3[u_1,u_3]$,  $\partial (u_7)= 4[u_1,u_5]+ 3[u_3,u_3] $ and \\  
\hspace*{5ex}$\partial (u_9)=a[[[u_1,u_2],u_2],[u_1,u_2]]+b(5[u_1,u_7]+10[u_3,u_5])$, \\
(ii) \ \ $\partial (u_3)=[u_1,u_1]$,  $\partial (u_5)=3[u_1,u_3]$,  $\partial (u_7)=4[u_1,u_5]+3[u_3,u_3] $ and \\  
\hspace*{5ex}$\partial (u_9)=[[[u_1,u_2],u_2],[u_1,u_2]]$, \\
(iii) \ \  $\partial (u_3)=[u_1,u_1]$,  $\partial (u_5)=3[u_1,u_3]$,  $\partial (u_7)=4[u_1,u_5]+3[u_3,u_3] $ and \\  
\hspace*{5ex}$\partial (u_9)=5[u_1,u_7]+10[u_3,u_5]$, \\
(iv) \  \  $\partial (u_3)=[u_1,u_1]$,  $\partial (u_5)=3[u_1,u_3]$,  $\partial (u_7)=4[u_1,u_5]+3[u_3,u_3] $ and 
$\partial (u_9)=0$.

\hfill  (Refer   p.94 of \cite{T} for the Quillen model of $ \C P^n$.)

Let a DGL-map $f:L(X)\to L(X)$ be given by 
$\ f(u_i)=\lambda_i u_i$ 
 for $\lambda_i\in\Q^*$ $(i=1,2,3,5,7,9)$. Then, $\mathcal{E}( X_0)/\mathcal{E}_{*}( X_0)$ is determined as follows:

\noindent 
The case (i):   $\lambda_3=\lambda_1^2$,   $\lambda_5=\lambda_1\lambda_3$, 
 $\lambda_7=\lambda_1\lambda_5=\lambda_3\lambda_3$ and 
 $\lambda_9=\lambda_1^2\lambda_2^3=\lambda_1\lambda_7=\lambda_3\lambda_5$
 from $\partial\circ f=f\circ \partial$.
If $f\in \calE(L (X))$, then we have $\lambda:=\lambda_1=\lambda_2\neq 0$.
Thus we have 
$\calE(L (X))/\mathcal{E}_{*}( L(X))\cong \Q^*$ with elements $f$ such that \ $f(u_1)=\lambda u_1$ and   
$$f(u_2)=\lambda u_2,   f(u_3)=\lambda^2 u_3,  f(u_5)=\lambda^3 u_5, f(u_7)=\lambda^4 u_7,  f(u_9)=\lambda^5 u_9.
$$

\noindent 
The case (ii):   $\lambda_3=\lambda_1^2$,   $\lambda_5=\lambda_1\lambda_3$, 
 $\lambda_7=\lambda_1\lambda_5=\lambda_3\lambda_3$ and 
 $\lambda_9=\lambda_1^2\lambda_2^3$
 from $\partial\circ f=f\circ \partial$.
Thus we have 
$\calE(L (X))/\mathcal{E}_{*}( L(X))=\{ (\lambda_1,\lambda_2)\}\cong {\Q^*}^{\times 2}$ with elements $f$ such that $f(u_1)=\lambda_1 u_1,  f(u_2)=\lambda_2 u_2$  and   
$$f(u_3)=\lambda_1^2 u_3,  f(u_5)=\lambda_1^3 u_5, f(u_7)=\lambda_1^4 u_7,  f(u_9)=\lambda_1^2\lambda_2^3 u_9.
$$

\noindent 
The case  (iii):   $\lambda_3=\lambda_1^2$,   $\lambda_5=\lambda_1\lambda_3$, 
 $\lambda_7=\lambda_1\lambda_5=\lambda_3\lambda_3$ and 
 $\lambda_9=\lambda_1\lambda_7=\lambda_3\lambda_5$
 from $\partial\circ f=f\circ \partial$.
Thus we have 
$\calE(L (X))/\mathcal{E}_{*}( L(X))=\{ (\lambda_1,\lambda_2)\}\cong {\Q^*}^{\times 2}$ with elements $f$ such that \ $f(u_1)=\lambda_1 u_1,  f(u_2)=\lambda_2 u_2$ and 
$$f(u_3)=\lambda_1^2 u_3,  f(u_5)=\lambda_1^3 u_5, f(u_7)=\lambda_1^4 u_7,  f(u_9)=\lambda_1^5 u_9.
$$

\noindent 
The case  (iv):   $\lambda_3=\lambda_1^2$,   $\lambda_5=\lambda_1\lambda_3$, 
 $\lambda_7=\lambda_1\lambda_5=\lambda_3\lambda_3$
 from $\partial\circ f=f\circ \partial$.
Thus  we have 
$\calE(L (X))/\mathcal{E}_{*}( L(X))=\{ (\lambda_1,\lambda_2,\lambda_9)\}\cong {\Q^*}^{\times 3}$ with elements $f$ such that \ $f(u_1)=\lambda_1 u_1,  f(u_2)=\lambda_2 u_2 ,  f(u_9)=\lambda_9 u_9$  and  
$$  f(u_3)=\lambda_1^2 u_3,  f(u_5)=\lambda_1^3 u_5, f(u_7)=\lambda_1^4 u_7.
$$

\end{example}

\begin{remark} \rm  \label{rem:exaNEXnonRatandRat}     
As for the relation between $N\calE (X)$ and $N\calE (X_0)$, we have the following examples: 
\begin{itemize}
\item[{\rm (1)}] \  $N\calE(S^{5} \times M(\Z /2 , 3)) = 5$, \  $N\calE( (S^{5} \times M(\Z /2 , 3))_{0}) = N\calE( S^{5}_{0}) = 5$.  
\item[{\rm (2)}] \  $N\calE(S^{5} \times M(\Z /2 , 7)) = 7$, \  $N\calE( (S^{5} \times M(\Z /2 , 7))_{0})= N\calE( S^{5}_{0})  = 5$. 
\item[{\rm (3)}] \  $N\calE(M(\Z /2 , m)) = m$, \  $N\calE(M(\Z /2 , m)_{0}) = N\calE(*) = 0$. 
\end{itemize} 
\end{remark}

\section{Self-closeness numbers defined by homology or cohomology  groups }  \label{sec:NstEX} 
Choi and Lee \cite{ChoiLee15} introduced the submonoid $\calA_{\sharp}^{k} (X)$ of $[X,X]$ making use of the homotopy groups of a space $X$.  
If we use the homology group $H_{i} (X)$ and the cohomology group $H^{i} (X)$ in place of the homotopy group  $\pi_{i} (X)$ in Definition \ref{def:selfcloseness}, we have definitions  of $\calA_{*}^{k} (X)$ and $\calA^{*}_{k} (X)$, respectively. 

\begin{definition}  \label{def:NstEX}  \rm 
The {\it homology self-closeness number $N_{*}\calE (X)$}  and the {\it cohomology self-closeness number $N^{*}\calE (X)$ of  a space   $X$} is defined by  
\[ N_{*}\calE (X) := \min \{\, k \ | \ \calA_{*}^{k} (X) = \calE (X) \} \ \ \mbox{and} \ \ N^{*}\calE (X) := \min \{\, k \ | \ \calA^{*}_{k} (X) = \calE (X) \}. \] 
\end{definition}

 If   $\calA_{*}^{k} (X) \neq  \calE (X)$ (resp. $\calA^{*}_{k} (X) \neq \calE (X)$) for any $k = 0,\, 1,\, 2,\, 3,\, \cdots$ or $\infty$, then $N_{*}\calE (X)$ (resp. $N^{*}\calE (X)$) is not defined by Definition \ref{def:NstEX}.  In this case we  define, temporarily,  $N_{*}\calE (X) = \O$ (resp. $N^{*}\calE (X) = \O$). In Definition  \ref{def:NstEX} the connectivity of the space $X$ is not assumed  to define  $N_{*}\calE (X)$ and  $N^{*}\calE (X)$. 

\begin{proposition} \label{prop:NstEXHomInv} 
If  $X$ and $Y$ have the same homotopy type, then the equalities $N_{*}\calE (X) = N_{*}\calE (Y)$ and $N^{*}\calE (X) = N^{*}\calE (Y)$ hold. 
\end{proposition}
\begin{proof} 
Assume that $X \simeq Y$ (homotopy equivalent). Then there exist maps $\varphi : X \to Y$ and $\psi : Y \to X$ such that $\psi \circ \varphi \simeq 1_{X}$  and  $\varphi \circ \psi \simeq 1_{Y}$.   

Assume that $N_{*}\calE (X) = k$ (resp. $N^{*}\calE (X) = k$) and  for a map   $f: Y \to Y$, the induced homomorphism $f_{*} : H_{i}(Y) \to  H_{i}(Y)$ (resp. $f^{*} : H^{i}(Y) \to  H^{i}(Y)$)  is an isomorphism for any $i \leq k$. Then the composite $\psi \circ f \circ \varphi : X \to X$ induces a homomorphism 
\[ \xymatrix@=20pt{ (\psi \circ f \circ \varphi)_{*} = \psi_{*} \circ f_{*} \circ \varphi_{*}:  H_{i}(X) \ar[r]_-{ \cong \  }^-{ \  \varphi_{*} \  } &  H_{i}(Y) \ar[r]_-{ \cong \ }^-{ \ f_{*}  \  }  & H_{i}(Y) \ar[r]_-{ \cong \ }^-{ \  \psi_{*} \  } &   H_{i}(X)  
} \] 
\[{\rm (resp.} \ \ \ \xymatrix@=20pt{ (\psi \circ f \circ \varphi)^{*} = \varphi^{*} \circ f^{*} \circ \psi^{*} :  H^{i}(X) \ar[r]_-{ \cong \ }^-{ \  \psi^{*} \  } &  H^{i}(Y) \ar[r]_-{ \cong \ }^-{ \ f^{*}  \  }  & H^{i}(Y) \ar[r]_-{ \cong \ }^-{ \  \varphi^{*} \  } &   H^{i}(X)  
}  {\rm )}\] 
which is an isomorphism  for any $i \leq k$. Hence $\psi \circ f \circ \varphi \in \calE(X)$. Let $g: X \to X$ be the homotopy inverse of $\psi \circ f \circ \varphi$, that is, 
\[g \circ (\psi \circ f \circ \varphi) \simeq 1_{X} \simeq (\psi \circ f \circ \varphi) \circ g. \]   
\[ \xymatrix@=20pt{
X  \ar[d]_{\varphi \ }  \ar@<0.5ex>[rr]^-{\psi \circ f \circ \varphi} &  & \ar@<0.5ex>[ll]^-{g} X \\
Y \ar[rr]_-{ f } \   &  &  Y \ar[u]_{ \  \psi} 
}  \] 
Let  $\overline{f} = \varphi \circ g \circ \psi$. Then we see 
\[ \overline{f} \circ f \simeq (\varphi \circ g \circ \psi) \circ f \circ (\varphi \circ \psi) =  \varphi \circ g \circ (\psi  \circ f \circ  \varphi) \circ \psi  \simeq  \varphi \circ 1_{X} \circ \psi =  \varphi \circ \psi  \simeq  1_{Y} ;  \] 
\[f \circ \overline{f} \simeq (\varphi \circ \psi) \circ f \circ (\varphi \circ g \circ \psi)  =  \varphi \circ (\psi \circ f \circ  \varphi) \circ g \circ \psi    \simeq  \varphi \circ 1_{X} \circ \psi =  \varphi \circ \psi  \simeq  1_{Y} . \] 
Therefore, we have  $f \in \calE(Y)$.  Hence $N_{*}\calE (Y) \leq k =N_{*}\calE (X)$. Similarly, we have $N_{*}\calE (X) \leq N_{*}\calE (Y)$.  It follows that $N_{*}\calE (X) = N_{*}\calE (Y)$. (resp. Hence $N^{*}\calE (Y) \leq k =N^{*}\calE (X)$. Similarly, we have $N^{*}\calE (X) \leq N^{*}\calE (Y)$.  It follows that $N^{*}\calE (X) = N^{*}\calE (Y)$.)
\end{proof}

\begin{remark} \rm  
\begin{enumerate} 
\item[{\rm (1)}] \ 
The inequalities $N_{*}\calE (X) \leq \dim (X)$ and $N^{*}\calE (X) \leq \dim (X)$ do not always hold for non-simply connected space $X$ as we see in the case $X=\R P^{2n}$ for $n \geq 1$:   If $n \geq 1$ and $k$ is odd and $k \neq \pm 1$, then $g_{k}^{2n} : \R P^{2n} \to \R P^{2n}$ in the proof of Theorem \ref{thm:NERPn}  induces an isomorphism for any $g_{k\sharp}^{2n} : \pi_{s}(\R P^{2n}) \to \pi_{s}(\R P^{2n})$ for any $s <2n$, and hence $g_{k*}^{2n} : H_{s}(\R P^{2n}) \to H_{s}(\R P^{2n})$ is an isomorphism  for any $s <2n$ by Proposition   \ref{prop:NEXeqNstEX}(1) and therefore an isomorphism  for any $s$. However,     $g_{k}^{2n}$ is not a homotopy equivalence.  
Therefore, we see $N_{*}\calE (\R P^{2n})= \O$ and $N^{*}\calE (\R P^{2n})= \O$ for any $n \geq 1$.  

If $X$ is not $1$-connected, then the condition that $f_{*} : H_{*}(X) \to H_{*}(X)$  (or $f^{*} : H^{*}(X) \to H^{*}(X)$) is an isomorphism does not always imply $f \in \calE(X)$; therefore, in this case,  $N_{*}\calE (X)= \O$ (or $N^{*}\calE (X)$= \O).

\item[{\rm (2)}]   \ 
If $X$ is $1$-connected  and $H_{k} (X)$  and $\pi_{k} (X)$ are finitely generated abelian groups for any $k$, then  we have $N_{*}\calE (X)  =    N\calE (X) \leq \dim (X)$ by  Corollary~\ref{cor:NEXeqNstEX} and Theorem 2 of Choi and  Lee \cite{ChoiLee15}.  

We also have a relation $N^{*}\calE (X) \leq  N_{*}\calE (X) + 1$ by Proposition \ref{prop:NcstEXleqNstEXp1}.  
\item[{\rm (3)}]   \   
The self-closeness number $N\calE (X)$  in Definition  \ref{def:selfcloseness} is homotopy invariant as we see by an argument similar to the proof of Proposition \ref{prop:NstEXHomInv}.  In Definition  \ref{def:selfcloseness} we assumed that $X$ is a $0$-connected space to define $N\calE (X)$. If $X$ is not $0$-connected, then $N\calE (X)$ is not defined when the components which do not contain base point have non-zero homotopy groups. 
If $X$ is a $0$-connected CW-complex, then $N\calE (X)$ is determined for some $k = 0,\,1,\,2,\, 3,\, \cdots$ or $\infty$ as is discussed in Choi and Lee \cite{ChoiLee15}. 

\end{enumerate} 
\end{remark}

\begin{example}  \label{exa:NstEXSnCPn} \rm 
\begin{enumerate} 
\item[{\rm (1)}] \  It is not difficult to prove:   
$N\calE(S^{n} )= N_{*}\calE (S^{n}) =N^{*}\calE (S^{n}) = n$ and $N\calE(\C P^{n} )= N_{*}\calE (\C P^{n}) =N^{*}\calE (\C P^{n}) = 2$.  
\item[{\rm (2)}]   \  If $X = \bigvee_{n=1}^{\infty} S^{n}$, then $ N_{*}\calE (X) = \infty$.   
\end{enumerate} 
\end{example}

Our method used in the discussion of Section \ref{sec:OdaYamaInprove} can be effectively applied in the section to prove some relations among $N\calE (X)$, $N_{*}\calE (X)$ and $N^{*}\calE (X)$. 

\begin{proposition}  \label{prop:NEXeqNstEX}  \  Let $n \geq 2$. Let $f: X \to X$ be a self-map. 
\begin{enumerate} 
\item[{\rm (1)}] \  If $X$ is a  $0$-connected space  and $f_{\sharp} : \pi_{i} (X) \to  \pi_{i} (X)$ \  is an isomorphism for any  $i \leq n$ and $H_{n} (X)$ is a finitely generated abelian group, then   $f_{*} : H_{i} (X)  \to H_{i} (X)$ \  is an isomorphism for any  $i \leq n$. 
\item[{\rm (2)}]   \   If $X$ is $1$-connected and $f_{*} : H_{i} (X) \to H_{i} (X)$ \  is an isomorphism for any  $i \leq n$ and $\pi_{n} (X)$ is a finitely generated abelian group, then   $f_{\sharp} : \pi_{i} (X)  \to \pi_{i} (X)$ \  is an isomorphism for any  $i \leq n$. 
\item[{\rm (3)}]   \ If  $f_{*} : H_{i} (X) \to H_{i} (X)$ \  is an isomorphism for any  $i \leq n$, then   $f^{*} : H^{i} (X)  \to H^{i} (X)$ \  is an isomorphism for any  $i \leq n$. 
\end{enumerate}   
\end{proposition}
\begin{proof} (1) \ If  $f_{\sharp} : \pi_{i} (X) \to \pi_{i} (X)$ \  is an isomorphism for any  $i \leq n$, then    $f_{*} : H_{i} (X)  \to H_{i} (X)$ \  is an isomorphism for any  $i \leq n -1$ and $f_{*} : H_{n} (X)  \to H_{n} (X)$ \  is surjective (by the Whitehead theorem, (7.13) Theorem of Whitehead \cite{Whitehead78})   and hence an isomorphism since $H_{n} (X)$ is a finitely generated abelian group.

\noindent 
(2)  \ If $X$ is $1$-connected and $f_{*} : H_{i} (X) \to H_{i} (X)$ \  is an isomorphism for any  $i \leq n$,  then   $f_{\sharp} : \pi_{i} (X)  \to \pi_{i} (X)$ \  is an isomorphism for any  $i \leq n-1$ and $f_{*} : H_{n} (X) \to H_{n} (X)$ \  is surjective (by the Whitehead theorem)   and hence an isomorphism since  $\pi_{n} (X)$ is a finitely generated abelian group.

\noindent 
(3) \  Suppose that  $f_{*} : H_{i}(X) \to H_{i}(X) $ is an isomorphism for any $i \leq n$. 
We have    a funtorial short exact sequence (the universal coefficient theorem): 
$$0 \xrightarrow{ \ \ } {\rm Ext} (H_{q-1} (X), \Z) \xrightarrow{ \ \ } H^{q} (X)\xrightarrow{ \   \ }   {\rm Hom} (H_{q} (X), \Z)  \xrightarrow{ \ \ }  0 $$ 
for any $q$ by  Corollary 53.2 (p.323) of Munkres \cite{Munkres84} or  3 Theorem (p.243) of Spanier \cite{Spanier66}.  By the naturality of the above  exact sequence  and the five lemma, we see that $f^{*} : H^{i}(X) \to H^{i}(X) $ is an isomorphism for any $i \leq n$.     
\end{proof}
 
In the following we assume that each of $N\calE (X)$, $N_{*}\calE (X)$ and $N_{*}\calE (X)$ are defined as $0,\, 1,\, 2,\, 3,\, \cdots$  or $\infty$ when we discuss equalities or inequalities among them.

\begin{theorem} \label{thm:NEXandNstExRel} 
\begin{itemize}
\item[{\rm (1)}]  \    $N\calE (X) \leq  N_{*}\calE (X)$ if $X$ is $0$-connected and $H_{k} (X)$ is a finitely generated abelian group for $k=N_{*}\calE (X)$.  
\item[{\rm (2)}] \   $N_{*}\calE (X)  \leq    N\calE (X)$ if $X$ is $1$-connected  and $\pi_{k} (X)$ is a finitely generated abelian group for $k=N\calE (X)$. 
\item[{\rm (3)}] \  $N_{*}\calE (X) \leq  N^{*}\calE (X) $ \ for any space $X$. 
\end{itemize} 
\end{theorem} 
\begin{proof} (1) \ Let $k=N_{*}\calE (X)$.  Assume that \ $f_{\sharp} : \pi_{i} (X) \to \pi_{i} (X)$ is an isomorphism for any  $i \leq  k$.  Then by Proposition  \ref{prop:NEXeqNstEX}(1), we see  $f_{*} : H_{i} (X)  \to H_{i} (X)$ \  is an isomorphism for any  $i \leq k$ since $H_{k} (X)$ is a finitely generated abelian group. It follows that \ $f \in \calE (X)$ and hence   $N\calE (X) \leq  N_{*}\calE (X)$.

\noindent 
(2)  \ Let $k=N\calE (X)$.  Assume that \ $f_{*} : H_{i} (X) \to H_{i} (X)$  is an isomorphism for any  $i \leq  k$.  Then by Proposition  \ref{prop:NEXeqNstEX}(2), we see  $f_{\sharp} : \pi_{i} (X)  \to \pi_{i} (X)$  \  is an isomorphism for any  $i \leq k$ since $\pi_{k} (X)$ is a finitely generated abelian group. It follows that \ $f \in \calE (X)$ and hence   $N_{*}\calE (X) \leq  N\calE (X)$.

\noindent 
(3)  \ Let $k=N^{*}\calE (X)$.  Assume that \ $f_{*} : H_{i} (X) \to H_{i} (X)$ is an isomorphism for any  $i \leq  k$.  Then by Proposition  \ref{prop:NEXeqNstEX}(3), we see  $f^{*} : H^{i} (X)  \to H^{i} (X)$ \  is an isomorphism for any  $i \leq k$. It follows that \ $f \in \calE (X)$ and hence   $N_{*}\calE (X) \leq  N^{*}\calE (X)$.  
\end{proof}

\begin{corollary} \label{cor:NEXeqNstEX} 
$N_{*}\calE (X)  =    N\calE (X)$ if $X$ is $1$-connected  and $H_{k} (X)$  and $\pi_{k} (X)$ are finitely generated abelian groups for any $k$. 
\end{corollary} 
\begin{proof} We have the result by  Parts (1) and (2) of Theorem  \ref{thm:NEXandNstExRel}.  
\end{proof}

\begin{proposition}  \label{prop:NcstEXleqNstEXp1} 
$N^{*}\calE (X) \leq  N_{*}\calE (X) + 1$ \  if $H_{q} (X)$ is finitely generated for any $q$.  
\end{proposition}
\begin{proof}  By 12 Theorem of Spanier \cite{Spanier66} (p.248), if  $H_{q} (X)$ is finitely generated abelian groups for any $q$, we have the following exact sequence:   
\[   0 \to {\rm Ext} (H^{q+1} (X), \Z) \to  H_{q} (X)    \to       {\rm Hom}(H^{q} (X), \Z)    \to 0. \] 
Let $k=N_{*}\calE (X)$.  Assume that \ $f^{*} : H^{i} (X)  \to H^{i} (X)$ is an isomorphism for any  $i \leq  k+1$.  Then by the above exact sequence, we see  $f_{*} : H_{i} (X)  \to H_{i} (X)$ \  is an isomorphism for any  $i \leq k$. It follows that \ $f \in \calE (X)$ and hence   $N^{*}\calE (X) \leq k + 1 =  N_{*}\calE (X) + 1$.  
\end{proof}

\begin{example}  \label{exa:SnwedMZ2} \rm  
Let  $n \geq 2$  and $X = S^{n} \vee M(\Z/2,n)$. Then we have $N\calE(X)  =N_{*}\calE(X) = n$ and $N^{*}\calE(X) = n + 1$. 
\end{example}

\begin{theorem}  \label{thm:NstEXleqNEX}    
\begin{enumerate} 
\item[{\rm (1)}]  \    If $X$ is $1$-connected and  $H_{k+1} (X) = 0$ for   $k = N\calE(X)$, then    $N_{*}\calE (X) \leq N\calE(X)$. 
\item[{\rm (2)}]   \  If  $X$ is $0$-connected and   $\pi_{k+1} (X) = 0$ for  $k  = N_{*}\calE (X)$, then $N\calE(X) \leq  N_{*}\calE (X)$.   
\item[{\rm (3)}]  \  Let $X$ be a space such that $H_i (X)$ is finitely generated for any $i$. If $H^{k+1} (X)$ is free for $k = N_{*}\calE(X)$, then  $N^{*}\calE (X)  \leq  N_{*}\calE(X)$. 
\end{enumerate}  
\end{theorem}  
\begin{proof}   
(1)  Assume that  $N\calE(X)=k$.   Suppose that $f:X \to X$ is a map such that $f_{*} : H_{i}(X) \to H_{i}(X) $ is an isomorphism for any $i \leq k$. By the Whitehead Theorem (7.13) on p.181 of \cite{Whitehead78}, we see that $f_{\sharp} : \pi_{i}(X) \to \pi_{i}(X) $ is an isomorphism for any $i \leq k$ (and epimorphism for $i = k+1$). It follows that $f$ is a homotopy equivalence. 

\noindent 
(2)   Assume that  $N_{*}\calE (X) = k$.    Suppose that $f:X \to X$ is a map such that $f_{\sharp} : \pi_{i}(X) \to \pi_{i}(X) $ is an isomorphism for any $i \leq k$. By the Whitehead theorem (7.13) of \cite{Whitehead78}, we see that $f_{*} : H_{i}(X) \to H_{i}(X) $ is an isomorphism for any $i \leq k$ (and epimorphism for $i = k+1$). It follows that $f$ is a homotopy equivalence. 

\noindent 
(3)  Assume that $N_{*}\calE(X)=k$.    Suppose that $f:X \to X$ is a map such that $f^{*} : H^{i}(X) \to H^{i}(X) $ is an isomorphism for any $i \leq k$. 
By 12 Theorem (p.248) of Spanier  \cite{Spanier66},   we have    a funtorial short exact sequence 
\[ 0 \xrightarrow{ \ \ } {\rm Ext} (H^{q+1} (X), \Z) \xrightarrow{ \ \ } H_{q} (X)\xrightarrow{ \ h \ }   {\rm Hom} (H^{q} (X), \Z)  \xrightarrow{ \ \ }  0 \]  
for any $q$. We see \ ${\rm Ext} (H^{k+1} (X), \Z) =0$,  since $H^{k+1} (X)$ is free.  It follows then that $f_{*} : H_{i}(X) \to H_{i}(X) $ is an isomorphism for any $i \leq k$ and hence $f$ is a homotopy equivalence. 
\end{proof}

Choi and  Lee proved the following relation in Theorem 3 of \cite{ChoiLee15}: 
\[ N\calE (X \times Y) \geq  \max \{ N\calE (X), N\calE (Y) \} . \] 
We can show the following relations for $N_{*}\calE (X \times Y)$ and $N^{*}\calE (X \times Y)$: 

\begin{proposition}  \label{prop:NstEXXtiY}  \  Let $X$ and $Y$ be any spaces.  Then,  
\begin{enumerate}  
\item[{\rm (1)}]   \  $N_{*}\calE (X \times Y) \geq  \max \{ N_{*}\calE (X), N_{*}\calE (Y) \}$.  
\item[{\rm (2)}]  \  $N^{*}\calE (X \times Y) \geq  \max \{ N^{*}\calE (X), N^{*}\calE (Y) \}$.  
\end{enumerate}   
\end{proposition}
\begin{proof}  (1) \  Let $N_{*}\calE (X \times Y) = k$ and $\max \{ N_{*}\calE (X), N_{*}\calE (Y) \} =  N_{*}\calE (X)$. 
Assume that a map $f: X \to X$ satisfies $f_{*} : H_{i} (X) \xrightarrow{\ \cong \ }  H_{i} (X)$ for any $i \leq  k$. Then 
$$(f \times 1_Y)_{*} : H_{i} (X \times Y) = \bigoplus_{0 \leq t \leq i} H_{i-t} (X) \otimes H_{t} (Y) \oplus \bigoplus_{0 \leq t \leq i -1} {\rm Tor} (H_{i-1-t} (X), H_{t} (Y))$$   
$$\xrightarrow{ \ \ \Phi  \ \ }  \bigoplus_{0 \leq t \leq i}  H_{i-t} (X) \otimes H_{i} (Y)  \oplus \bigoplus_{0 \leq t \leq i -1} {\rm Tor} (H_{i-1-t} (X), H_{t} (Y)) = H_{i} (X \times Y)$$  
is an isomorphism for any $i \leq  k$, where $\Phi =\oplus (f_{*} \otimes  1_{Y *}) \oplus  {\rm Tor}  (f_{*} ,  1_{Y *})$. 
Then there exists a map $G: X \times Y \to X \times Y$ such that $(f \times 1_Y) \circ G \simeq  1_{ X \times Y} \simeq  G \circ (f \times 1_Y)$. 

Let $i_{X} : X \to X \times Y$ be the inclusion and $p_{X} : X \times Y  \to X$ the projection. Consider the following commutative diagram:  
\[\xymatrix{
X \times Y  \ar[rr]^{f \times 1_Y} 
&  & X \times Y   \ar[rr]^{G}
&  &  X \times Y  \ar[d]^{ \   p_X } \ar[rr]^{f \times 1_Y}  &  &  X \times Y \ar[d]^{ \   p_X }  \\ 
X \ar[u]^{i_X \ } \ar[rr]^{f}&  &  X  \ar[u]^{i_X \ }  \ar[rr]^{p_X \circ G \circ i_X}
&  &   X  \ar[rr]^{f} &  &  X 
}\] 
Since $(f \times 1_Y) \circ  i_X =  i_X \circ f$ \  and \ $p_X  \circ   (f \times 1_Y)  =  f  \circ p_X$,   we have 
\[ (p_X \circ G \circ i_X) \circ f =  p_X \circ G \circ (f \times 1_Y) \circ  i_X \simeq    p_X \circ  1_{ X \times Y} \circ  i_X = 1_{X} ;   \] 
\[ f   \circ (p_X \circ G \circ i_X)=  p_X  \circ   (f \times 1_Y) \circ G \circ i_X  \simeq    p_X \circ  1_{ X \times Y} \circ  i_X = 1_{X} .  \] 
Therefore, $f \in \calE (X)$ \ and hence $N_{*}\calE (X) \leq k = N_{*}\calE (X \times Y)$.  
\vspace{1ex}

\noindent  
 (2) \  Let $N^{*}\calE (X \times Y) = k$ and $\max \{ N^{*}\calE (X), N^{*}\calE (Y) \} =  N^{*}\calE (Y)$. 
Assume that a map $f: Y \to Y$ satisfies $f^{*} : H^{i} (Y) \xrightarrow{\ \cong \ }   H^{i} (Y)$ for any $i \leq  k$. Then 
$$(1_X \times f)^{*} : H^{i} (X \times Y) = \bigoplus_{0 \leq t \leq i}  {\rm Hom} (H_{i-t} (X) , H^{t} (Y)) \oplus \bigoplus_{0 \leq t \leq i -1} {\rm Ext} (H_{i-1-t} (X), H^{t} (Y))$$   
$$\xrightarrow{ \ \   \Psi  \ \ }  \bigoplus_{0 \leq t \leq i} {\rm Hom} ( H_{i-t} (X) , H^{i} (Y) ) \oplus \bigoplus_{0 \leq t \leq i -1} {\rm Ext} (H_{i-1-t} (X), H^{t} (Y)) = H^{i} (X \times Y)$$   
is an isomorphism  for any $i \leq  k$, where $\Psi = {\rm Hom} (1_{X  *},  f^{*}) \oplus  {\rm Ext}  (1_{X  *},  f^{*})$. 
Then there exists a map $G: X \times Y \to X \times Y$ such that $(1_X \times f) \circ G \simeq  1_{ X \times Y} \simeq  G \circ (1_X \times f)$. Then by an argument similar to Part (1), we have $N^{*}\calE (Y) \leq k = N^{*}\calE (X \times Y)$.

In the formula of Exercise 3.6.1 (p.90) of Weibel \cite{Weibel94}: 
\[ 0 \to  \bigoplus_{p+q = n -1} {\rm Ext} (H_{p} (P), H^{q} (Q))   \to H^{n} ({\rm Hom} (P,Q)) \to  \bigoplus_{p+q = n}  {\rm Hom} (H_{p} (P), H^{q} (Q))  \to 0  ,  \]  
if we use the singular chain complex  $P = S_{*} (X)$ and  the singular cochain complex $Q= {\rm Hom} (S_{*} (Y), \Z)$  and the  isomorphism    
\[ {\rm Hom} (S_{*} (X) \otimes S_{*} (Y), \Z)  \cong  {\rm Hom} (S_{*} (X), {\rm Hom}( S_{*} (Y), \Z)), \] 
then we have the above K\"unneth formula. 
\end{proof}


\begin{thebibliography}{99}

\bibitem{AM} M. Arkowitz  and K. Maruyama, Self homotopy equivalences which induce the identity on homology, cohomology or 
homotopy groups, Toplogy  Appl. \textrm{87} (1998), 133--154.


\bibitem{Ba} I. K. Babenko and I. A. Taimanov,
{On nonformal  simply connected   symplectic manifolds},
Sibirsk. Math. Z. \textrm{41} (2000), 253--269.

\bibitem{BBFI95} Y. G. Borisovich,  N. M. Bliznyakov, T. N. Fomenko and Y. A. Izrailevich, \textrm{Introduction to differential and algebraic topology}  (Translated and revised from the 1980 Russian edition), Kluwer Texts in  Math. Sci. 9,  Kluwer Academic Publishers Group,  Dordrecht, 1995.


\bibitem{ChoiLee15}  H. W. Choi and K. Y. Lee,     
       Certain numbers on the groups of self-homotopy equivalences, Topology  Appl. \textrm{181} (2015),  104--111.      



\bibitem{DGMS} 
P. Deligne, P. Griffiths, J.  Morgan and D.  Sullivan,  Real  homotopy  theory  of  K\"{a}hler manifolds, Invent. Math.   \textrm{29}    (1975),  245--274. 


\bibitem{FHT} Y. F\'{e}lix, S. Halperin and J. C. Thomas, \textrm{Rational homotopy theory},  Graduate texts in Math.  205, Springer-Verlag, New York Heidelberg Berlin, 2001.

\bibitem{FOT}
Y.~F\'{e}lix, J. Oprea  and  D. Tanr\'{e},   \textrm{Algebraic models in geometry},  Oxford Graduate texts in Math.  17, 2008.



\bibitem{GM} P. Griffiths and J. Morgan,
 \textrm{Rational homotopy theory and differential forms}, Progress in Math. 16, Birkh${\ddot{\rm a}}$user,  Boston,   1981.

\bibitem{Ha}  S. Halperin,  Finiteness in the minimal models of Sullivan, Trans. Amer. Math. Soc.   \textrm{230}  (1977),  173--199.


\bibitem{Hatcher01}  A. Hatcher,  \textrm{Algebraic Topology},  Cambridge University Press,  2001. 


\bibitem{HMR} P. Hilton, G. Mislin and J. Roitberg, \textrm{Localization of nilpotent groups and spaces}, North-Holland Math. Studies  15,   1975.

\bibitem{HS} S. Halperin and J. Stasheff,  Obstructions to homotopy equivalences, Adv. in Math.   \textrm{32}  (1979),  233--279.  

\bibitem{LupOprea95} G. Lupton and J. Oprea,  Cohomologically symplectic spaces: toral actions and the Gottlieb group,  
Trans. Amer. Math. Soc. \textrm{347} (1995), 261--288.  

\bibitem{Milnor63} J. Milnor, \textrm{Morse theory},   Based on lecture notes by M.  Spivak and R. Wells,   Annals of Mathematics Studies  51,  Princeton University Press,  Princeton  N.J.,  1963. 

\bibitem{Milnor59}  J. Milnor,  \textrm{On spaces having the homotopy type of a CW-complex}, Trans. Amer. Math. Soc.   \textrm{90}  (1959),  272--280. 

\bibitem{MimuraToda71} M. Mimura and H. Toda,
  On $p$-equivalences and $p$-universal spaces, Comment. Math. Helv.   \textrm{46}  (1971), 87--97.

\bibitem{Munkres84}  J. R. Munkres,  \textrm{Elements of algebraic topology}, Addison-Wesley  Publishing Company, Menlo Park, CA,  1984.  


\bibitem{NM} J. Neisendorfer and T. Miller, Formal and coformal spaces,  Illinois J.  Math. \textrm{22} (1978),  565--580

 
\bibitem{OdaYama17}   N. Oda  and T. Yamaguchi,   {Self-homotopy equivalences and cofibrations},  Topology  Appl.    228  (2017),  341--354.  

\bibitem{OdaYama18}   N. Oda  and T. Yamaguchi,   {Self-maps of spaces in fibrations},  Homology, Homotopy  Appl.     20  (2018),  289--313. 

\bibitem{Olum53}  P.  Olum,  {Mappings of manifolds and the notion of degree},  Ann. of Math.   \textrm{58} (1953),  458--480. 

\bibitem{Q} D. Quillen,   Rational homotopy theory, Ann. Math.     \textrm{90} (1969), 205--295.


\bibitem{Spanier66}  E. H. Spanier, \textrm{Algebraic topology}, Springer-Verlag, New York, 1966. 


\bibitem{S} D. Sullivan,  Infinitesimal computations in topology,
I.H.E.S.  \textrm{47}  (1978),  269--331.

\bibitem{T}  D. Tanr\'{e},   \textrm{Homotopie Rationnelle:
Mod\`{e}les de Chen, Quillen, Sullivan}, Lecture Notes in Math.   1025, Springer,   1983. 


\bibitem{tomDieck}  T. tom Dieck, \textrm{Algebraic Topology}, EMS Textbooks in Math.,  European Math. Soc., Z\"urich,  2008. 


\bibitem{Vick95} J. W. Vick,   \textrm{Homology theory. An introduction to algebraic topology}  Second edition,   Graduate Texts in Math.  145,  Springer-Verlag, New York, 1994. 

\bibitem{Weibel94}  C. A. Weibel,  \textrm{An introduction to homological algebra},  Cambridge Studies in Advanced Math.   38,   Cambridge University Press,  Cambridge,  1994.   

\bibitem{Whitehead78}  G. W. Whitehead,  \textrm{Elements of homotopy theory},   Graduate texts in Math. 61,  Springer-Verlag, New York Heidelberg Berlin, 1978. 
\end{thebibliography}
\end{document}